\documentclass[letterpaper,twoside]{article}

\usepackage[letterpaper,left=1.5in,right=1.5in,top=1in,bottom=1in]{geometry}

\newcommand{\myauthor}{Benjamin Antieau and Lennart Meier}
\newcommand{\mytitle}{The Brauer group of the moduli stack of elliptic curves}

\author{Benjamin Antieau\footnote{Benjamin Antieau was supported
by NSF Grants DMS-1461847 and DMS-1552766.}~ and Lennart Meier\footnote{Lennart Meier was
supported by DFG SPP 1786.}}
\title{\mytitle}
\date{}

\usepackage{amsmath}
\usepackage{amscd}
\usepackage{amsbsy}
\usepackage{amssymb}
\usepackage{verbatim}
\usepackage{eufrak}
\usepackage{eucal}
\usepackage{microtype}
\usepackage{mathrsfs}
\usepackage{amsthm}
\usepackage[all,cmtip]{xy}
\usepackage{tikz}
\usetikzlibrary{matrix,arrows}

\usepackage[pdfstartview=FitH,
            pdfauthor={\myauthor},
            pdftitle={\mytitle},
            colorlinks,
            linkcolor=reference,
            citecolor=citation,
            urlcolor=e-mail,
            backref]{hyperref}

\usepackage[abbrev,lite,alphabetic]{amsrefs}
\usepackage{dsfont}

\usepackage{fancyhdr}

\usepackage{color}
\definecolor{todo}{rgb}{1,0,0}
\definecolor{conditional}{rgb}{0,1,0}
\definecolor{e-mail}{rgb}{0,.40,.80}
\definecolor{reference}{rgb}{.20,.60,.22}
\definecolor{mrnumber}{rgb}{.80,.40,0}
\definecolor{citation}{rgb}{0,.40,.80}

\let\oldmarginpar\marginpar
\renewcommand\marginpar[1]{\-\oldmarginpar[\raggedleft\footnotesize #1]%
{\raggedright\footnotesize #1}}
\newcommand{\df}[1]{{\bf #1}}

\newcommand{\Ascr}{\mathcal{A}}
\newcommand{\Bscr}{\mathcal{B}}
\newcommand{\Fscr}{\mathcal{F}}
\newcommand{\Hscr}{\mathcal{H}}
\newcommand{\Lscr}{\mathcal{L}}
\newcommand{\Mscr}{\mathcal{M}}
\newcommand{\MM}{\mathcal{M}}
\newcommand{\Oscr}{\mathcal{O}}
\newcommand{\Xscr}{\mathcal{X}}
\newcommand{\Yscr}{\mathcal{Y}}

\newcommand{\B}{\mathrm{B}}
\newcommand{\E}{\mathrm{E}}
\newcommand{\F}{\mathrm{F}}
\renewcommand{\H}{\mathrm{H}}
\newcommand{\M}{\mathrm{M}}
\newcommand{\R}{\mathrm{R}}

\renewcommand{\AA}{\mathds{A}}
\newcommand{\CC}{\mathds{C}}
\newcommand{\FF}{\mathds{F}}
\newcommand{\GG}{\mathds{G}}

\newcommand{\PP}{\mathds{P}}
\newcommand{\QQ}{\mathds{Q}}
\newcommand{\Q}{\mathds{Q}}
\newcommand{\RR}{\mathds{R}}
\newcommand{\ZZ}{\mathds{Z}}
\newcommand{\Z}{\mathds{Z}}

\newcommand{\ZZf}[1]{\ZZ[\tfrac{1}{#1}]}

\newcommand{\op}{\mathrm{op}}

\DeclareMathOperator{\tors}{tors}
\DeclareMathOperator{\id}{id}
\DeclareMathOperator{\Tr}{Tr}
\newcommand{\legendre}[2]{\genfrac{(}{)}{}{}{#1}{#2}}

\newcommand{\OO}{\mathcal{O}}

\DeclareMathOperator{\Pic}{Pic}
\DeclareMathOperator{\Br}{Br}

\DeclareMathOperator{\Hom}{Hom}

\DeclareMathOperator{\PGL}{PGL}

\DeclareMathOperator{\GL}{GL}

\newcommand{\Gm}{\mathds{G}_{m}}

\newcommand{\et}{\mathrm{\acute{e}t}}
\newcommand{\Zar}{\mathrm{Zar}}
\newcommand{\pl}{\mathrm{pl}}

\newcommand{\ind}{\mathrm{ind}}

\newcommand{\tr}{\mathrm{tr}}
\newcommand{\res}{\mathrm{res}}
\newcommand{\coker}{\mathrm{coker}}
\newcommand{\tensor}{\otimes}

\DeclareMathOperator{\Spec}{Spec}

\newcommand{\ShEnd}{\mathcal{E}\mathrm{nd}}

\newcommand{\iso}{\cong}

\numberwithin{equation}{section}
\theoremstyle{plain}
\newtheorem{theorem}[equation]{Theorem}
\newtheorem*{theorem*}{Theorem}
\newtheorem{lemma}[equation]{Lemma}

\newtheorem{proposition}[equation]{Proposition}

\newtheorem{corollary}[equation]{Corollary}

\newtheoremstyle{named}{}{}{\itshape}{}{\bfseries}{.}{.5em}{#1 \thmnote{#3}}
\theoremstyle{named}

\theoremstyle{definition}
\newtheorem{definition}[equation]{Definition}

\newtheorem{example}[equation]{Example}

\newtheorem{construction}[equation]{Construction}

\theoremstyle{remark}
\newtheorem{remark}[equation]{Remark}

\begin{document}

\maketitle

\begin{abstract}
    \noindent
    We compute the Brauer group of $\Mscr_{1,1}$, the moduli stack of elliptic
    curves, over $\Spec\ZZ$, its localizations, finite fields of odd
    characteristic, and algebraically closed fields of characteristic not $2$. The methods involved include the use of the parameter space
    of Legendre curves and the moduli stack $\Mscr(2)$ of curves with full (naive) level $2$
    structure, the study of the Leray--Serre spectral sequence
    in \'etale cohomology and the Leray spectral sequence in fppf cohomology,
    the computation of the group cohomology of $S_3$ in a certain integral
    representation, the
    classification of cubic Galois extensions of $\QQ$, the computation of
    Hilbert symbols in the ramified case for the primes $2$ and $3$, and
    finding $p$-adic elliptic curves with specified properties.


    \paragraph{Mathematics Subject Classification 2010.}
    Primary:
    \href{http://www.ams.org/mathscinet/msc/msc2010.html?t=14Fxx&btn=Current}{14F22},
    \href{http://www.ams.org/mathscinet/msc/msc2010.html?t=14Fxx&btn=Current}{14H52},
    \href{http://www.ams.org/mathscinet/msc/msc2010.html?t=14Fxx&btn=Current}{14K10}.
    Secondary:
    \href{http://www.ams.org/mathscinet/msc/msc2010.html?t=11Gxx&btn=Current}{11G05},
    \href{http://www.ams.org/mathscinet/msc/msc2010.html?t=11Gxx&btn=Current}{11G07}.
\end{abstract}

\tableofcontents

\section{Introduction}

Brauer groups of fields have been considered since the times of Brauer and
Noether; later Grothendieck generalized Brauer groups to the case of arbitrary
schemes. Although both the definition via Azumaya algebras and the
cohomological definition generalize to arbitrary Deligne--Mumford stacks,
Brauer groups of stacks have so far mostly been neglected especially for
stacks containing arithmetic information. Some exceptions are
the use of Brauer groups of root stacks, as in the work of Chan--Ingalls~\cite{chan-ingalls} on the
minimal model program for orders on surfaces, the work of
Auel--Bernardara--Bolognesi~\cite{auel-bernardara-bolognesi} on derived  categories of families of quadrics, and
the work of
Lieblich \cite{lieblich-arithmetic}, who computed the Brauer group
$\Br(\B\mu_n)$ over a field and applied it to the period-index problem. In this
paper, we study the Brauer group $\Br(\Mscr)$ of the moduli stack of elliptic
curves $\Mscr=\Mscr_{1,1}$.

The case of the Picard group has been considered before. Mumford showed
in~\cite{mumford-picard} that $\Pic(\Mscr_k)=\H^1_{\et}(\Mscr_k,\Gm)\iso\ZZ/12$ if $k$ is a field of
characteristic not dividing $6$, and that the Picard
group is generated by the Hodge bundle $\lambda$. The bundle $\lambda$ is characterized by the property
that $u^*\lambda\iso p_*\Omega^1_{E/T}$ when $u:T\rightarrow\Mscr$ classifies a
family $p:E\rightarrow T$ of elliptic curves. This calculation was extended by Fulton
and Olsson who showed that $\Pic(\Mscr_S)\iso\Pic(\AA^1_S)\oplus\ZZ/12$ whenever $S$
is a reduced scheme~\cite{fulton-olsson}.

In contrast, an equally uniform description of $\Br(\Mscr_S)$ does not seem
possible (even if we assume that $S$ is regular noetherian); both the result
and the proofs depend much more concretely on the arithmetic on $S$. The
following is a sample of our results in ascending order of difficulty.
We view (5) as the main result of this paper.

\begin{theorem}\label{thm:intromain}
    \begin{enumerate}
        \item[\emph{(1)}] $\Br(\Mscr_k) = 0$ if $k$ is an algebraically closed field of
        characteristic not $2$,\footnote{Minseon Shin has proved in~\cite{shin} that when $k$ is algebraically closed of characteristic $2$ one
    has $\Br(\Mscr_k)\iso\ZZ/2$.}
        \item[\emph{(2)}] $\Br(\Mscr_k) \cong \Z/12$ if $k$ is a finite field of characteristic not $2$,
        \item[\emph{(3)}] $\Br(\Mscr_\QQ)\cong \Br(\QQ) \;\oplus\bigoplus_{p\not\equiv
        3\bmod 4}\ZZ/4\;\oplus\bigoplus_{p\equiv 3\bmod 4}\ZZ/2\; \oplus\;\H^1(\QQ, C_3),$ \\
        where $p$ runs over all primes and $-1$,
        \item[\emph{(4)}] $\Br(\Mscr_{\ZZf{2}})\iso \Br(\ZZf{2})\oplus
            \ZZ/2\oplus\ZZ/4$, and
        \item[\emph{(5)}] $\Br(\Mscr) = 0$.
    \end{enumerate}
\end{theorem}

In all cases the non-trivial classes can be explicitly described via cyclic
algebras (see Lemma \ref{lem:3section}, Proposition \ref{prop:2cyclic} and
Remark \ref{rem:4cyclic}). In general, the $p$-primary torsion for $p\geq 5$ is
often easy to control via the following theorem.

\begin{theorem}
    Let $S$ be a regular noetherian scheme and $p\geq 5$ prime.
    Assume that $S[\frac1{2p}] = S_{\ZZf{2p}}$ is dense in $S$ and that $\Mscr_S \to S$ has a section.
    Then, the natural map $_p\Br'(S)\rightarrow{_p\Br'(\Mscr_{S})}$ is an isomorphism.
\end{theorem}

Here, $\Br'(\Mscr_S)$ denotes the cohomological Brauer group, which agrees with
$\Br(\Mscr_S)$ whenever $S$ is affine and at least one prime is invertible on
$S$.

Let us now explain how to compute the $2$- and $3$-primary torsion in the example of $\Br(\Mscr_{\ZZf{2}})$. 
We use the $S_3$-Galois cover
$\Mscr(2)\rightarrow\Mscr_{\ZZf{2}}$, where $\Mscr(2)$ is the moduli stack of
elliptic curves with full (naive) level $2$-structure. The Leray--Serre spectral
sequence reduces the problem of computing $\Br(\Mscr_{\ZZf{2}})$
to understanding the low-degree $\Gm$-cohomology of $\Mscr(2)$ \emph{together}
with the action of $S_3$ on these cohomology groups and to the computation of differentials.

To understand the groups themselves, it is sufficient to use the fact that
$\Mscr(2)\iso\B C_{2,X}$, the classifying stack of cyclic $C_2$-covers over
$X$, where $X$ is the parameter space of Legendre curves. Explicitly, $X$ is
the arithmetic
surface $$X=\PP^1_{\ZZf{2}}-\{0,1,\infty\}=\Spec\ZZ[\tfrac12,t^{\pm 1},(t-1)^{-1}],$$ and the universal
Legendre curve over $X$ is defined by the equation 
$$y^2=x(x-1)(x-t).$$
The Brauer group of $\B C_{2,X}$ can be described using a Leray--Serre spectral
sequence as well, but this description is not $S_3$-equivariant, which causes
some complications, and we have to use the $S_3$-equivariant map from $\Mscr(2)$ to
$X$, its coarse moduli space, to get full control. Knowledge about the Brauer group of $X$ leads for the $3$-primary torsion to the following conclusion.

\begin{theorem}
    Let $S$ be a regular noetherian scheme. If $6$ is a unit on $S$, then
    there is an exact sequence
    $$0\rightarrow{_3\Br'(S)}\rightarrow{_3\Br'(\Mscr_S)}\rightarrow\H^1(S,C_3)\rightarrow 0,$$ which is non-canonically split.
\end{theorem}

There is a unique cubic Galois extension of $\QQ$ which is ramified at most at $(2)$ and $(3)$, namely $\QQ(\zeta_9+\overline{\zeta}_9)$. This shows that the cokernel of $\Br(\ZZf{6})\hookrightarrow\Br(\Mscr_{\ZZf{6}})$ is
$\ZZ/3$.
The proof that this extra class does not extend to $\Mscr_{\ZZf{2}}$, which is
similar to the strategy discussed below for the $2$-torsion, uses the
computation of cubic Hilbert symbols at the prime $3$.
Putting these ingredients together, we conclude that
$\Br(\Mscr_{\ZZf{2}})$ is a $2$-group, and further computations first over $\ZZ[\tfrac12,i]$ and then over $\ZZf{2}$
let us deduce the structure in Theorem~\ref{thm:intromain}(4).
The corresponding general results on $2$-torsion are contained in Proposition
\ref{prop:mi} and Theorem \ref{thm:brm}. They are somewhat more complicated to
state so we omit them from this introduction.

To show that $\Br(\Mscr)=0$, we need an extra argument since all our arguments
using the Leray--Serre spectral sequence presuppose at least that $2$ is inverted. Note first that the map $\Br(\Mscr) \to \Br(\Mscr_{\ZZf{2}})$ is an injection. Thus, we have only to show that the non-zero classes in $\Br(\Mscr_{\ZZf{2}})$ do not
extend to $\Mscr$. Our general method is the following. For each non-zero class
$\alpha$ in the Brauer group of $\Mscr_{\ZZf{2}}$, we exhibit an elliptic curve
over $\Spec\ZZ_2$ such that the restriction of $\alpha$ to
$\Spec\QQ_2$ is non-zero. Such an elliptic curve defines a morphism $\Spec \ZZ_2 \to \Mscr$ and we obtain a commutative diagram
\begin{equation}\label{eq:BrauerExtending}
\xymatrix{
    \Br(\QQ_2)  &   \Br(\Mscr_{\ZZf{2}})\ar[l]\\
    \Br(\ZZ_2)\ar[u]    &   \Br(\Mscr)\ar[u]\ar[l]
}
\end{equation} 
Together with the fact that $\Br(\ZZ_2)=0$, this diagram implies that the class $\alpha$ cannot come
from $\Br(\Mscr)$. This argument requires us to
understand explicit generators for $\Br(\Mscr_{\ZZf{2}})$ and the computation of Hilbert symbols again. 

We remark that the computation of $\Br(\Mscr)$ -- while important in algebraic geometry and in the arithmetic of
elliptic curves -- was nevertheless originally motivated by considerations in chromatic
homotopy theory, especially in the possibility in constructing twisted
forms of the spectrum $\mathrm{TMF}$ of topological modular
forms. The Picard group computations of Mumford and Fulton--Olsson are primary
inputs into the computation of the Picard group of $\mathrm{TMF}$ due to Mathew
and Stojanoska~\cite{mathew-stojanoska}.

\paragraph{Conventions.} We will have occasion to use Zariski, \'etale, and
fppf cohomology of schemes and Deligne-Mumford stacks. We will denote these by
$\H^i_{\Zar}(X,F)$, $\H^i(X,F)$, and $\H^i_{\pl}(X,F)$ when $F$ is an
appropriate sheaf on $X$. Note in particular that without other adornment,
$\H^i(X,F)$ or $\H^i(R,F)$ (when $X=\Spec R$) always denotes \'etale
cohomology. If $G$ is a group and $F$ is a $G$-module, we let $\H^i(G,F)$ denote the group
cohomology.

For all stacks $\Xscr$ appearing in this paper, we will have
$\Br'(\Xscr)\iso\H^2(\Xscr,\Gm)$. Thus, we will use the two groups
interchangeably. When working over a general base $S$, we will typically state
our results in terms of $\Br'(S)$ or $\Br'(\Mscr_S)$. However, when working
over an affine scheme, such as $S=\Spec\ZZf{2}$, we will write $\Br(S)$ or
$\Br(\Mscr_S)$. There should be no confusion as in all of these cases we will
have $\Br(S)=\Br'(S)$ and $\Br(\Mscr_S)=\Br'(\Mscr_S)$ and so on by~\cite{dejong-gabber}.

For an abelian group $A$ and an integer $n$, we will denote by $_nA$ the subgroup of $n$-primary torsion
elements: $_nA=\{x\in A:n^kx=0\text{ for some $k\geq 1$}\}$.

\paragraph{Acknowledgments.} We thank the Hausdorff Research Institute for
Mathematics, UIC, and Universit\"at Bonn for hosting one or the other author in
2015 and 2016 while this paper was being written. Furthermore, we thank Asher
Auel, Akhil Mathew, and Vesna Stojanoska for their ideas
at the beginning of this project and Peter Scholze and Yichao Tian for the
suggestion to use fppf-cohomology and the idea for the splitting in Proposition
\ref{prop:fppfsplitting}. Finally, we thank the anonymous referees for their
helpful suggestions.

\section{Brauer groups, cyclic algebras, and ramification}\label{sec:BrauerCyclicRam}

We review here some basic facts about the Brauer group, with special attention
to providing references for those facts in the generality of Deligne-Mumford
stacks. For more details about the Brauer group in general,
see~\cite{grothendieck-brauer-1}. 

Any Deligne-Mumford stack has an associated \'etale topos, and we
can therefore consider \'etale sheaves and \'etale cohomology~\cite{lmb}.

\begin{definition}
    If $X$ is a quasi-compact and quasi-separated Deligne-Mumford stack, the
    \df{cohomological Brauer group} of $X$ is defined to be
    $\Br'(X)=\H^2(X,\Gm)_{\tors}$, the torsion subgroup of $\H^2(X,\Gm)$.
\end{definition}

Because of its definition as the torsion in a cohomology group, the
cohomological Brauer group is amenable to computation via Leray--Serre spectral
sequences, long exact sequences, and so on, as we will see in the next sections.
However, our main interest is in the \emph{Brauer group} of Deligne-Mumford
stacks.

\begin{definition}
    An \df{Azumaya algebra} over a Deligne-Mumford stack $X$ is a sheaf of
    quasi-coherent $\Oscr_X$-algebras $\Ascr$ such that $\Ascr$ is
    \'etale-locally on $X$ isomorphic to $\Mscr_n(\Oscr_X)$, the sheaf of $n\times
    n$-matrices over $\Oscr_X$, for some $n\geq 1$.
\end{definition}

In particular, an Azumaya algebra $\Ascr$ is a locally free $\Oscr_X$-module, and the degree $n$ appearing
in the definition is a locally constant function. If the degree $n$ is in fact
constant, then $\Ascr$ corresponds to a unique $\PGL_n$-torsor on $X$ because
the group of $k$-algebra automorphisms of $\M_n(k)$ is isomorphic to $\PGL_n(k)$ for
fields $k$. The exact sequence
$1\rightarrow\Gm\rightarrow\GL_n\rightarrow\PGL_n\rightarrow 1$ gives a
boundary map $$\delta:\H^1(X,\PGL_n)\rightarrow\H^2(X,\Gm).$$ For an Azumaya
algebra $\Ascr$ of degree $n$, we write $[\Ascr]$ for the class $\delta(\Ascr)$ in
$\H^2(X,\Gm)$. In general, when $X$ has
multiple connected components, its invariant $[\Ascr]\in\H^2(X,\Gm)$ is
computed on each component.

\begin{example}
    \begin{enumerate}
        \item   If $E$ is a vector bundle on $X$ of rank $n>0$, then $\Ascr=\ShEnd(E)$, the
            sheaf of endomorphisms of $E$, is an Azumaya algebra on $X$. Indeed,
            in this case, $\Ascr$ is even Zariski-locally equivalent to
            $\Mscr_n(\Oscr_X)$. The class of $\Ascr$ in $\H^1(X,\PGL_n)$ is the
            image of $E$ via $\H^1(X,\GL_n)\rightarrow\H^1(X,\PGL_n)$, so
            the long exact sequence in nonabelian cohomology implies that
            $[\Ascr]=0$ in $\H^2(X,\Gm)$.
        \item   If $\Ascr$ and $\Bscr$ are Azumaya algebras on $X$, then
            $\Ascr\otimes_{\Oscr_X}\Bscr$ is an Azumaya algebra.
    \end{enumerate}
\end{example}

\begin{definition}
    Two Azumaya algebras $\Ascr$ and $\Bscr$ are \df{Brauer equivalent} if
    there are vector bundles $E$ and $F$ on $X$ such that
    $\Ascr\otimes_{\Oscr_X}\ShEnd(E)\iso\Bscr\otimes_{\Oscr_X}\ShEnd(F)$.
    The \df{Brauer group} $\Br(X)$ of a Deligne-Mumford stack $X$ is the multiplicative monoid of isomorphism classes of
    Azumaya algebras under tensor product modulo Brauer equivalence.
\end{definition}

In terms of Azumaya algebras, addition is given by the tensor product
$[\Ascr]+[\Bscr]=[\Ascr\otimes_{\Oscr_X}\Bscr]$, and
$-[\Ascr]=[\Ascr^{\op}]$. Here are the basic structural facts we will use about the Brauer group.

\begin{proposition}\label{prop:BrOmnibus}
    \begin{enumerate}
        \item[{\rm (i)}] The Brauer group of a quasi-compact and quasi-separated
            Deligne-Mumford stack $X$ is the subgroup of $\Br'(X)$ generated by $[\Ascr]$
            for $\Ascr$ an Azumaya algebra on $X$.
        \item[{\rm (ii)}]    The Brauer group $\Br(X)$ is a torsion group for any
            quasi-compact and quasi-separated Deligne-Mumford stack $X$.
        \item[{\rm (iii)}]   If $X$ is a regular and noetherian Deligne-Mumford stack, then $\H^2(X,\Gm)$ is
            torsion, so in particular $\Br'(X)\iso\H^2(X,\Gm)$.
        \item[{\rm (iv)}]   If $X$ is a regular and noetherian Deligne-Mumford
            stack, and if $U\subseteq X$ is a dense
            open subset, then
            the restriction map $\H^2(X,\Gm)\rightarrow\H^2(U,\Gm)$ is injective.
        \item[{\rm (v)}]    If $X$ is a scheme with an ample line bundle,
            then $\Br(X)=\Br'(X)$.
        \item[{\rm (vi)}] If $X$ is a regular and noetherian scheme with $p$
            invertible on $X$, then the morphism
            $_p\H^i(X,\Gm) \to {_p\H^i(\AA^1_X, \Gm)}$ on $p$-primary torsion is an
            isomorphism for all $i\geq 0$. 
    \end{enumerate}
\end{proposition}

\begin{proof}
    See~\cite{grothendieck-brauer-1}*{Section 2} for points (i) and (ii).
    The proof of (iv) is analogous to that of \cite{lieblich}*{Lemma~3.1.3.3}, using an analogue of \cite{lmb}*{Proposition~15.4} to generalize \cite{lieblich}*{Lemma~3.1.1.9} to algebraic stacks. 
    See also~\cite{auel-bernardara-bolognesi}*{Proposition~1.26}.
    For (v), see de Jong~\cite{dejong-gabber}.


    For (iii), see for instance~\cite{grothendieck-brauer-2}*{Proposition~1.4}
    in the case of schemes. We must generalize it to the case of a regular noetherian Deligne-Mumford stack $X$. We can assume that
    $X$ is connected and hence irreducible as $X$ is normal. Pick $U\subseteq
    X$ a dense open such that $U$ admits a \emph{finite} \'etale map
    $V\rightarrow U$ of degree $n$ where $V$ is a scheme. The composition
    $\H^2(U,\Gm)\rightarrow\H^2(V,\Gm)\rightarrow\H^2(U,\Gm)$ of restriction
    and transfer is multiplication by $n$. Since $\H^2(V,\Gm)$ is torsion, this
    implies that $\H^2(U,\Gm)$ is torsion. By (iv), $\H^2(X,\Gm)$ is torsion as well. 
    
    Finally, we have to prove (vi). For $i\leq 1$, the $\AA^1$-invariance is even true before taking $p$-power torsion (see e.g.\ \cite[II.6.6 and II.6.15]{hartshorne}). Because $p$ is invertible on $S$, the maps
    $\H^i(X,\mu_{p^m})\rightarrow\H^n(\AA^1_X,\mu_{p^m})$ are isomorphisms for
    all $i$ and $m$. This is the $\AA^1$-invariance of \'etale
    cohomology and a proof can be found in \cite{milne}*{Corollary VI.4.20}. The short exact sequence
    $$ 1 \to \mu_{p^m} \to \Gm \xrightarrow{p^m} \Gm \to 1$$
    induces short exact sequences
    \[\xymatrix{
        0 \ar[r] & \H^{i-1}(X, \Gm)/p^m \ar[r]\ar[d] & \H^i(X, \mu_{p^m}) \ar[r]\ar[d]^\cong & \H^i(X,\Gm)[p^m] \ar[d] \ar[r] & 0 \\
        0 \ar[r] & \H^{i-1}(\AA^1_X, \Gm)/p^m \ar[r] & \H^i(\AA^1_X, \mu_{p^m}) \ar[r] & \H^i(\AA^1_X,\Gm)[p^m] \ar[r] & 0 
    }\]
    Inductively, using the $\AA^1$-invariance of $\H^{i-1}(X,\Gm)$ if $i\leq 2$ or
    of $_p\H^{i-1}(X,\Gm)$ if $i>2$ as well as the fact that $\H^{i-1}(X,\Gm)$
    is torsion for $i\geq 3$ since $X$ is regular and noetherian
    (\cite{grothendieck-brauer-2}*{Proposition~1.4}),
    we see by the five lemma that $\H^i(X, \Gm)[p^m] \to\H^i(\AA^1_X,
    \Gm)[p^m]$ is an isomorphism for all $i$ and $m$ and hence we also get an
    isomorphism $_p \H^i(X, \Gm) \to {_p \H^i(\AA^1_X, \Gm)}$.
\end{proof}

\begin{remark}\label{rem:Gabber}
    At a couple points, we use another important fact, due to Gabber, which
    says that if $p:Y\rightarrow X$ is a surjective finite locally free map, if $\alpha\in\Br'(X)$, and
    if $p^*\alpha\in\Br(Y)$, then $\alpha\in\Br(X)$. This is already proved in
    Gabber~\cite{gabber}*{Chapter II, Lemma 4} for locally ringed topoi with strict hensel
    local rings, so we need to add nothing further in our setting.
\end{remark}

By far the most important class of Azumaya algebras arising in arithmetic
applications is the class of cyclic algebras. For a treatment over fields,
see~\cite{gille-szamuely}*{Section 2.5}. These algebras give a concrete
realization of the cup product in fppf cohomology
$$\H^1_{\pl}(\Xscr,C_n)\times\H^1_{\pl}(\Xscr,\mu_n)\rightarrow\H^2_{\pl}(\Xscr,\mu_n)\rightarrow\H^2_{\pl}(\Xscr,\Gm)$$
for an algebraic stack $\Xscr$.
Given that $C_n$ and $\Gm$ are smooth and that the image of such a cup product
is torsion, we can re-write this as
$\H^1(\Xscr,C_n)\times\H^1_{\pl}(\Xscr,\mu_n)\rightarrow\H^2_{\pl}(\Xscr,\mu_n)\rightarrow\Br'(\Xscr)$.
Given $\chi\in\H^1(\Xscr,C_n)$ and
$u\in\H^1_{\pl}(\Xscr,\mu_n)$, we write $[(\chi,u)_n]$ or $[(\chi,u)]$ for the image of the cup product in
$\Br'(\Xscr)$.

Fix an algebraic stack $\Xscr$.
Let $p(\chi):\Yscr\rightarrow\Xscr$ be the cyclic Galois cover defined by
$\chi\in\H^1(\Xscr,C_n)$. Then, $p(\chi)_*\Oscr_\Yscr$ is a locally free
$\Oscr_\Xscr$-algebra of
finite rank which comes equipped with a canonical $C_n$-action. There is a natural isomorphism
$\mathbf{Spec}_\Xscr\left(p(\chi)_*\Oscr_\Yscr\right)\iso \Yscr\rightarrow \Xscr$.

The group $\H^1_{\pl}(\Xscr,\mu_n)$
fits into a short exact sequence
\begin{equation}\label{eq:h1fppf}
0\rightarrow\Gm(\Xscr)/n\rightarrow\H^1_{\pl}(\Xscr,\mu_n)\rightarrow\Pic(\Xscr)[n]\rightarrow 0.
\end{equation}
It is helpful to have a more concrete description of $\H^1_{\pl}(\Xscr,\mu_n)$, which will
also show that the exact sequence \eqref{eq:h1fppf} is non-canonically split.  Let $\H(\Xscr,n)$ 
be the abelian group of equivalence classes of pairs $(\Lscr,s)$ where
$\Lscr\in\Pic(\Xscr)[n]$
and $s$ is a choice of trivialization $s:\Oscr_\Xscr\rightarrow\Lscr^{\otimes n}$.
Two pairs $(\Lscr,s)$ and $(\Mscr,t)$ are equivalent if there is an isomorphism
$g:\Lscr\rightarrow\Mscr$ and a unit $v\in\Gm(\Xscr)$ such that $g(s)=v^nt$.
The group structure is given by tensor product of line bundles and of
trivializations. The following construction is part of Kummer theory and is well-known in the scheme case (see e.g.\ \cite{milne}*{p.\ 125}, which also shows part of Proposition \ref{prop:fppfsplitting} below).

\begin{construction}\label{const:root}
    Let $\Xscr$ be an algebraic stack and fix a class $[u]\in\H(\Xscr,n)$ with $u = (\Lscr,s)$ for a line
    bundle $\Lscr$ on $\Xscr$ with
    a trivialization $s\colon \OO_{\Xscr} \to \Lscr^{\tensor n}$. Define
    $\Oscr_{\Xscr}(\sqrt[n]{u})$ as $\bigoplus_{i\in\ZZ}\Lscr^{\otimes
    i}/(s-1)$ and
    $\Xscr(\sqrt[n]{u})\rightarrow\Xscr$ as
    the affine morphism 
    $\mathbf{Spec}_\Xscr \Oscr_{\Xscr}(\sqrt[n]{u})\rightarrow\Xscr.$ 
    It is easy to see that $\Xscr(\sqrt[n]{u}) \to \Xscr$
    is an fppf $\mu_n$-torsor and that this construction defines a group
    homomorphism $\H(\Xscr,n)\rightarrow\H^1_{\pl}(\Xscr,\mu_n)$. Indeed, if $\Xscr = \Spec A$ and $\Lscr$ is
    trivial, then $\Xscr(\sqrt[n]{u}) \cong \Spec A(\sqrt[n]{s^{-1}})\iso\Spec
    A(\sqrt[n]{s})$.
\end{construction}

\begin{proposition}\label{prop:fppfsplitting}
    The map $\H(\Xscr,n)\rightarrow\H^1_{\pl}(\Xscr,\mu_n)$ is an
    isomorphism. In particular, there is a non-canonical splitting 
    $$\H^1_{\pl}(\Xscr, \mu_n) \cong \Gm(\Xscr)/n \oplus \Pic(\Xscr)[n].$$
\end{proposition}

\begin{proof}
    We claim that the line bundle associated with the $\mu_n$-torsor
    $\Xscr(\sqrt[n]{u}) \to \Xscr$ (via the map
    $\H^1_{\pl}(\Xscr,\mu_n)\rightarrow\H^1(\Xscr,\Gm)$ induced by the  inclusion $\mu_n \to \Gm$) is
    exactly $\Lscr$. Indeed, the obvious map from $\Xscr(\sqrt[n]{u})$ to
    $\mathbf{Spec}_\Xscr\left(\bigoplus_{i\in\Z} \Lscr^{\tensor i}\right)$ is
    equivariant along the inclusion $\mu_n \to \Gm$ and the target is the
    $\Gm$-torsor associated with $\Lscr$. Thus, the composition
    $\H(\Xscr,n)\rightarrow\H^1_{\pl}(\Xscr,\mu_n)\rightarrow\Pic(\Xscr)[n]$ is
    surjective. By~\eqref{eq:h1fppf}, to prove that
    $\H(\Xscr,n)\rightarrow\H^1_{\pl}(\Xscr,\mu_n)$ is an isomorphism, it suffices to prove that the induced map
    from the kernel of this map to $\Gm(\Xscr)/n$ is an isomorphism.
    But, this follows immediately from the definition of $\H(\Xscr,n)$.
    
    It remains to construct the splitting. By Pr{\"u}fer's theorem \cite[Theorem 17.2]{fuchs}, $\Pic(\Xscr)[n]$ is a direct
    sum of cyclic groups. Thus, we have only to show that for every divisor $k$ of
    $n$ and each $k$-torsion element
    $[\Lscr]$ in $\Pic(\Xscr)$, there exists a pre-image in
    $\H^1_{\pl}(\Xscr,\mu_n)$ that is $k$-torsion. This pre-image can be constructed as
    follows: choose a trivialization $s\colon \OO_\Xscr \to \Lscr^{\tensor k}$ and
    take the $\mu_n$-torsor associated with the $\mu_k$-torsor
    $\Xscr(\sqrt[k]{v}) \to \Xscr$ with $v=(\Lscr,s)$.
\end{proof}

Now, given $\chi$ and $u$ as above, we let $\widetilde{\Ascr}_{\chi,u}$ be the coproduct
$$\left(p(\chi)_*\Oscr_\Yscr\coprod_{\Oscr_\Xscr}\Oscr_\Xscr(\sqrt[n]{u})\right)$$ in the category of
sheaves of quasi-coherent (associative and unital) $\Oscr_\Xscr$-algebras. Finally, we let
$$\Ascr_{\chi,u}=\widetilde{\Ascr}_{\chi,u}/(ab-ba^{g(\chi)})$$ 
be the quotient of
$\widetilde{\Ascr}_{\chi,u}$ by the two-sided ideal generated by terms
$ab-ba^{g(\chi)}$ where $a$ is a local section of $\chi_*\Oscr_Y$, $b$ is a local
section of $\Oscr_X(\sqrt[n]{u})$, and $a^{g(\chi)}$ denotes the action of
$g(\chi)$ on $\chi_*\Oscr_\Yscr$ for $g(\chi)$
a fixed generator of the action of $C_n$ on $p(\chi)_*\Oscr_\Yscr$.

\begin{lemma}
    Given an algebraic stack $\Xscr$ and classes $\chi\in\H^1(\Xscr,C_n)$ and
    $u\in\H^1_{\pl}(\Xscr,\mu_n)$, the algebra $\Ascr_{\chi,u}$ is an Azumaya
    algebra on $\Xscr$.
\end{lemma}

\begin{proof}
    It suffices to check this
fppf-locally, and in particular we can assume that in fact $\Xscr$ is a
scheme $X$, that $\Lscr\cong \OO_{\Xscr}$, and that $\chi$ classifies a $C_n$-Galois cover
$p(\chi)\colon Y\rightarrow X$.
In this case, we can write the $\Oscr_X$-algebra $p(\chi)_*\Oscr_Y$ Zariski locally as a quotient $\Oscr_X[x]/f(x)$ for some monic
polynomial $f(x)$ of degree $n$. Then, locally, we
have that 
$$\Ascr_{\chi,u}\iso\Oscr_X\langle
x,y\rangle/(f(x),y^n-u,xy-yx^{g(\chi)}),$$
a quotient of the free algebra over $\Oscr_X$ on generators $x$ and $y$. Note that the sections $x^iy^j$ for $0 \leq i,j \leq n-1$ form a basis of $\Ascr_{\chi,u}$ as an $\OO_X$-module and in particular that $\Ascr_{\chi,u}$ is (locally) a free $\OO_X$-module. 
Examining the fibers of $\Ascr_{\chi,u}$ over $X$, we
obtain the usual definition of a cyclic algebra given
in~\cite{gille-szamuely}*{Proposition 2.5.2}. So, $\Ascr_{\chi,u}$ is locally free with central
simple fibers and~\cite{grothendieck-brauer-1}*{Th\'eor\`eme~5.1} implies
that $\Ascr_{\chi,u}$ is Azumaya. 
\end{proof}

The following proposition is well-known, but we do not know an exact reference.
However, in the case of quaternion algebras, it is given
in~\cite{parimala-srinivas}*{Lemma~8}.

\begin{proposition}\label{prop:cyclicpresentation}
    Let $X$ be a regular noetherian scheme and suppose that
    $\chi\in\H^1(X,C_n)$ and $u\in\H^1_{\pl}(X,\mu_n)$ are fixed classes. In the notation above,
    we have $[(\chi,u)_n]=[\Ascr_{\chi,u}]$ in $\Br'(X)$. In particular
    $[(\chi,u)_n]\in\Br(X)$.
\end{proposition}

\begin{proof}
    We can assume that $X$ is connected. As $\Br'(X) \to \Br'(K)$ is injective (see \cite[IV.2.6]{milne} or Proposition~\ref{prop:BrOmnibus}(iv)), it is enough to check this on a
    generic point $\Spec K$ of $X$.
    By definition and the previous example,
    $\Ascr_{\chi,u}$ is a standard cyclic algebra over $K$ as
    defined in~\cite{gille-szamuely}*{Chapter 2}. They check
    in~\cite{gille-szamuely}*{Proposition~4.7.3} that $\Ascr_{\chi,u}$ does indeed have
    Brauer class given by the cup product. See also the remark at the beginning
    of the proof of~\cite{gille-szamuely}*{Proposition~4.7.1}.
\end{proof}

\begin{remark}
    The reader may notice that $\Ascr_{\chi,u}$ is defined in complete
    generality, but that we only prove the equality
    $[(\chi,u)_n]=[\Ascr_{\chi,u}]$ for regular noetherian schemes. In fact,
    this equality extends to arbitrary algebraic stacks, but a different
    argument is necessary. It is given at the end of Section~\ref{section:BCm}.
\end{remark}

We will abuse notation and write $(\chi,u)_n$ or even just $(\chi,u)$ for
$\Ascr_{\chi,u}$. This is called a \df{cyclic algebra}. If there is a primitive
$n$th root of unity $\omega \in \mu_n(\Xscr)$ and the cyclic Galois cover
$\Yscr\to\Xscr$
is obtained by adjoining an $n$th root of an element $a\in \Gm(\Xscr)$, we write
$(a,u) = (a,u)_{\omega}$ for the corresponding cyclic algebra, where we need
the choice of $\omega$ to fix an isomorphism $\H^1(\Xscr,C_n)\iso\H^1(\Xscr,\mu_n)$. For $n=2$, we obtain the classical notion
of a quaternion algebra.

For us, the key
point about the cyclic algebra is that it allows us to compute the ramification
of a Brauer class explicitly. Before explaining this, we mention that by the Gabber--\v{C}esnavi\v{c}ius purity theorem the Brauer group of a regular noetherian scheme $X$ is insensitive to throwing away high
codimension subschemes.

\begin{proposition}[Purity~\cite{gabber}*{Chapter I},~\cite{cesna}]\label{prop:purity}
    Let $X$ be a regular noetherian scheme. If
    $U\subseteq X$ is a
    dense open subscheme with complement of codimension at least $2$, then the
    restriction map $\Br'(X)\rightarrow\Br'(U)$ is an isomorphism.
\end{proposition}


Let $X$ be a regular noetherian scheme and let $\eta$ be the scheme of generic points
in $X$. The purity theorem reduces the problem of computing $\Br(X)$ from
$\Br(\eta)$ to the problem of extending Brauer classes $\alpha\in\Br(\eta)$ 
over divisors in $X$. This is controlled by ramification theory.
The following proposition is basically well-known, but we include a proof for
the reader's convenience.

\begin{proposition}\label{prop:puritycalc}
    Let $X$ be a regular noetherian scheme, $i\colon D\subseteq X$ a Cartier divisor that is
    regular with complement $U$, and $n$ an integer. There is an
    exact sequence
    \begin{equation}\label{eq:ramseq}
        0\rightarrow{_n\Br'(X)}\rightarrow{_n\Br'(U)}\xrightarrow{\mathrm{ram}_D}
        {_n}\H^3_D(X, \Gm) \to {_n}\H^3(X, \Gm) \to {_n}\H^3(U,\Gm).
    \end{equation}
    If $n$ is prime to the residue characteristics of $X$, we
    have ${_n}\H^3_D(X, \Gm) \cong \H^1(D,{_n\QQ/\ZZ})$.
\end{proposition}

\begin{proof}
    By \cite[6.1]{grothendieck-brauer-3} or \cite[III.1.25]{milne} there is a long exact sequence
    $$\H^2(X, \Gm) \rightarrow \H^2(U,\Gm)  \to \H^3_D(X, \Gm) \to\H^3(X, \Gm) \to\H^3(U,\Gm).$$
    By \cite[Propsition 1.4]{grothendieck-brauer-2} all occurring groups are
    torsion so that the sequence is still exact after taking $n$-primary
    torsion. Furthermore, by Proposition \ref{prop:BrOmnibus} the first map is
    an injection.
    
    We may assume that $n=p$ is a prime, in which case $_p\QQ/\ZZ\iso\QQ_p/\ZZ_p$.
    We have to show that $_p\H^3_D(X,\Gm)\iso\H^1(D,\QQ_p/\ZZ_p)$.
    By either the
    relative cohomological purity theorem of
    Artin~\cite{sga4-3}*{Th\'eor\`eme~XVI.3.7 and 3.8} (when both $X$ and $D$ are
    smooth over some common base scheme $S$) or the absolute
    cohomological purity theorem of Gabber~\cite{fujiwara}*{Theorem~2.1},
    we have the following identifications of local cohomology sheaves:
    $\Hscr_{D}^t(\mu_{p^\nu})=0$ for $t\neq 2$ and
    $\Hscr_{D}^2(\mu_{p^\nu})\iso i_*\ZZ/{p^\nu}(-1)$. It follows from the
    long exact sequence of local cohomology sheaves associated to the exact
    sequence
    $1\rightarrow\mu_{p^\nu}\rightarrow\Gm\xrightarrow{p^\nu}\Gm\rightarrow 1$
    that $p$ acts invertibly on $\Hscr^t_{D}(\Gm)$ for $t\neq 1,2$.
    Moreover, since $X$ is regular and noetherian, for every open $V\subset X$,
    the map $\Pic(V) \to \Pic(U\cap V)$ is surjective with kernel $(i_*\ZZ)(V)$
    by \cite[II.6.5]{hartshorne} and $\Br'(V) \to \Br'(U\cap V)$ is injective;
    thus $\Hscr^2_{D}(\Gm)=0$ and $\Hscr^1_{D}(\Gm) \cong i_*\ZZ$. Therefore,
    the only contribution to $p$-primary torsion in $\H^3_{D}(X,\Gm)$ in the
    local to global spectral sequence
    $$\H^s(X,\Hscr^t_{D}(\Gm))\Rightarrow\H^{s+t}_{D}(X,\Gm)$$
    is $_p\H^2(X,i_*\ZZ)$. We obtain 
    $$_p\H^3_D(X,\Gm) \cong {_p}\H^2(X,\Hscr^1_{D}(\Gm))\iso{_p\H^2(D,\ZZ)}\iso\H^1(D,\QQ_p/\ZZ_p)$$
    as desired, where the last isomorphism holds because $\H^i(D,\QQ)=0$ for $i>0$ since
    $D$ is normal (see for example~\cite{deninger}*{2.1}).
\end{proof}

Note that in all cases where we use Proposition~\ref{prop:puritycalc}, the easier
relative cohomological purity theorem of Artin is applicable, so that in the
end our paper does not rely on the more difficult results of Gabber and \v{C}esnavi\v{c}ius.

We will need to know a special case of the ramification map
$\mathrm{ram}_D:\Br(U)\rightarrow\H^1(D,\QQ/\ZZ)$.

\begin{proposition}\label{prop:ramification}
    Let $R$ be a discrete valuation ring with fraction field $K$ and residue
    field $k$. Set $X=\Spec R$, $U=\Spec K$, and $x=\Spec k$. Let
    $(\chi,\pi)_n$ be a cyclic algebra over $K$, where $\chi$ is a degree $n$ cyclic
    character of $K$, $\pi$ is
    a uniformizing parameter of $R$ (viewed as an element of $\Gm(K)/n$), and $n$ is prime to the characteristic of
    $k$. Finally, let $L/K$ be the cyclic Galois extension
    defined by $\chi$. If the integral closure $S$ of $R$ in $L$ is a discrete
    valuation ring with uniformizing parameter $\pi_S$, then
    $\mathrm{ram}_{(\pi)}(\chi,\pi)$ is the class of the cyclic extension $S/(\pi_S)$
    over $k$.
\end{proposition}

\begin{proof}
    See~\cite{saltman}*{Lemma~10.2}.
\end{proof}

Finally, we discuss cyclic algebras over local fields and some implications for
global calculations. Let $K$ be a local field containing a primitive $n$th root of unity $\omega$. Then there is a pairing
$$\binom{-,-}{\mathfrak{p}}:\Gm(K)/n \times \Gm(K)/n \to \mu_n(K),$$
called the \df{Hilbert symbol} (where $\mathfrak{p}$ stands for the maximal ideal of the ring of integers of $K$). Our standard reference for this pairing is \cite[Section V.3]{neukirch}. If $\mathfrak{p}$ is generated by an element $\pi$, we will also write $\binom{a,b}{\pi}$. We will use Hilbert symbols to check whether explicitly defined cyclic algebras are zero in the Brauer group. 

\begin{proposition}\label{prop:Hilbert}
    For $a,b \in K^\times$, the cyclic algebra $(a,b)_{\omega}$ is trivial in $\Br(K)$ if and only if $\binom{a,b}{\mathfrak{p}} = 1$. 
\end{proposition}

\begin{proof}
     By Proposition V.3.2 of \cite{neukirch} the Hilbert symbol
     $\binom{a,b}{\mathfrak{p}}$ equals $1$ if and only if $a$ is a norm from
     the extension $K(\sqrt[n]{b})|K$. By \cite[Corollary
     4.7.7]{gille-szamuely}, this happens if and only if $(a,b)_{\omega}$
     splits, i.e.\ defines the trivial class in $\Br(K)$.
\end{proof}

More generally, local class field theory calculates $\Br(\eta)$ when $\eta=\Spec K$
where $K$ is a (non-archimedean) local field. Let $X=\Spec R$ and $x=\Spec k$, where $R$ is the
ring of integers in $K$ and $k$ is the residue field of $R$.
As $\H^3(X,\Gm)\iso\H^3(x,\Gm)=0$ (for instance
by~\cite{grothendieck-brauer-3}*{Th\'eor\`eme~1.1}),
we find from~\cite{grothendieck-brauer-3}*{Corollaire~2.2} that there is an
exact sequence $$0\rightarrow\Br(X)\rightarrow\Br(\eta)\rightarrow\H^1(x,\QQ/\ZZ)\rightarrow
0.$$ The idea is similar to that of Proposition~\ref{prop:puritycalc}, but here
the proof is easier as $0\rightarrow\Gm\rightarrow j_*\Gm\rightarrow
i_*\ZZ\rightarrow 0$ is exact where $j:\eta\rightarrow X$ and $i:x\rightarrow
X$. Since $K$ is local, $k$ is
finite, so that $\H^1(x,\QQ/\ZZ)\iso\QQ/\ZZ$. However, since $R$ is Henselian,
$\Br(X)=\Br(x)$ (see~\cite{grothendieck-brauer-1}*{Corollaire~6.2}), and $\Br(x)=0$ by a theorem of
Wedderburn (see~\cite{grothendieck-brauer-3}*{Proposition~1.5}).

Now, let $K$ be a number field, and let $R$ be a localization of the ring of
integers of $K$. Set $\eta=\Spec K$ and $X=\Spec R$. In this
case, by~\cite{grothendieck-brauer-3}*{Proposition~2.1}, there is an exact sequence
$$0\rightarrow\Br(X)\rightarrow\Br(\eta)\rightarrow\bigoplus_{\mathfrak{p}\in
X^{(1)}}\Br(\Spec K_{\mathfrak{p}}),$$ where $X^{(1)}$ denotes the set of
codimension $1$ points of $X$. This exact sequence is compatible
with~\eqref{eq:ramseq} and with
the exact sequence
\begin{equation}\label{eq:cft}
0\rightarrow\Br(\eta)\rightarrow\bigoplus_{\mathfrak{p}}\Br(\Spec K_{\mathfrak{p}})\rightarrow\QQ/\ZZ\rightarrow
0\end{equation}
of class field theory (see~\cite{neukirch-schmidt-wingberg}*{Theorem~8.1.17}). The sum ranges over the finite and the
infinite places of $K$, and the map $\Br(\Spec K_\mathfrak{p})\rightarrow\QQ/\ZZ$ is
the isomorphism described above when $\mathfrak{p}$ is a finite place, the
natural inclusion $\ZZ/2\rightarrow\QQ/\ZZ$ when $K_{\mathfrak{p}}\iso\RR$, and
the natural map $0\rightarrow\QQ/\ZZ$ when $K_{\mathfrak{p}}\iso\CC$.
Using these sequences, we can compute the Brauer group of $X$. 

The two fundamental observations we need about~\eqref{eq:cft} are that a
class $\alpha\in\Br(\eta)$ is ramified at no fewer than $2$ places and that if
$K$ is purely imaginary, then $\alpha\in\Br(\eta)$ is ramified at no fewer than $2$
\emph{finite} places. The reader can easily verify the following examples.

\begin{example}
    \begin{enumerate}
        \item[(1)] $\Br(\ZZ)=0$.
        \item[(2)] $\Br(\ZZ[\tfrac1p])\iso\ZZ/2$.
        \item[(3)] $\Br(\ZZ[\tfrac{1}{pq}])\iso\ZZ/2\oplus\QQ/\ZZ$.
        \item[(4)] $\Br(\ZZ[\tfrac1p,\zeta_p])=0$.
    \end{enumerate}
\end{example}

We will use these computations and those like them throughout the paper, often
without comment.

\section{The low-dimensional $\Gm$-cohomology of $\B C_m$}\label{section:BCm}

Let $S$ be a scheme. Write $C_{n,S}$ for the constant \'etale group scheme on the
cyclic group $C_n$ of order $n\geq 2$ over $S$. We will often suppress the base
in the notation and simply write $C_n$ when the base is clear from context. The purpose of this section is
to make a basic computation of the $\Gm$-cohomology of the Deligne-Mumford stack $\B C_n=\B C_{n,S}$. In fact, we
are only interested in the cases $n=2$ and $n=4$, but the general case is no more
difficult.

The first tool for our computations of the
\'etale cohomology of an \'etale sheaf $\Fscr$ is the convergent
Leray--Serre spectral
sequence. If $\pi:Y\rightarrow X$ is a $G$-Galois cover where $X$ and $Y$ are
Deligne-Mumford stacks, then this spectral sequence has the form
$$\E^{p,q}_2=\H^p(G,\H^q(Y,\pi^*\Fscr))\Rightarrow\H^{p+q}(X,\Fscr),$$ with
differentials $d_r$ of bidegree $(r,1-r)$.

We will use the spectral sequence in this section for the $C_n$-Galois cover $\pi\colon S\rightarrow\B C_n$, where $C_n$ acts
trivially on $S$. In this case it is of the form 
$$\E_2^{p,q}=\H^p(C_n,\H^q(S,\Gm))\Rightarrow\H^{p+q}(\B C_{n,S},\Gm)$$
and 
Figure~\ref{fig:bc2} displays the low-degree part of its $\E_2$-page.
The fact that $C_n$ acts trivially on the cohomology of $S$ implies that the left-most column is
simply the $\Gm$-cohomology of $S$. For $F$ a constant $C_n$-module, we use the standard isomorphisms
$\H^i(C_n,F)\iso F[n]$ when $i>0$ odd, and $\H^i(C_n,F)\iso F/n$ when $i>0$ is
even.
We are only interested in $\H^i(\B C_n,\Gm)$ for $0\leq i\leq 2$.

\begin{figure}[h]
    \centering
    \begin{equation*}
        \xymatrix@R=3pt{
            \H^2(S,\Gm)& & & \\
            \Pic(S)&\Pic(S)[n]\ar[rrd]^{}&\Pic(S)/n& \\
            \Gm(S)&\mu_n(S)&\Gm(S)/n&\mu_n(S)\\
        }
    \end{equation*}
    \caption{The $\E_2$-page of the Leray--Serre spectral sequence computing $\H^i(\B C_n,\Gm)$.}
    \label{fig:bc2}
\end{figure}

Recall Grothendieck's theorem that the natural morphism
$\H^i(X,G)\rightarrow\H^i_{\pl}(X,G)$ is an isomorphism for $i\geq 0$ when $G$
is a smooth group scheme on $X$ (such as $C_n$ or $\Gm$). See \cite[Section 5, Th\`eor\'eme 11.7]{grothendieck-brauer-3}.
Note that this implies the agreement of \'etale and fppf cohomology on a Deligne--Mumford stack. For example, Grothendieck's theorem implies that the morphism from the Leray--Serre spectral
sequence above to the analogous Leray--Serre spectral sequence
$$\E_2^{p,q}=\H^p(C_n,\H^q_{\pl}(S,\Gm))\Rightarrow\H^{p+q}_{\pl}(\B C_n,\Gm)$$
for the fppf cohomology is an isomorphism; thus the comparison map
$$\H^i(\B C_n, \Gm) \to \H^i_{\pl}(\B C_n, \Gm)$$
is also an isomorphism. For $\Gm$-coefficients, we will thus not distinguish between
\'etale and fppf cohomology in what follows. 

We use these observations to compute the Picard and Brauer groups of $\B C_{n,S}$ via a Leray spectral sequence.
The idea is borrowed from \cite{lieblich-arithmetic}*{Section 4.1}.
Consider the map to the coarse moduli space $c\colon \B C_{n,S} \to S$. We claim that
\begin{align*}
    \R^0_{\pl}c_*\Gm &= \Gm \\
    \R^1_{\pl}c_*\Gm &= \mu_n \\
    \R^2_{\pl}c_*\Gm &= 0.
\end{align*}Indeed, $\R^i_{\pl}c_*\Gm$ is the fppf-sheafification of $U
\mapsto\H^i(\B C_{n,U},\Gm)$ and in the Leray--Serre spectral sequences all classes in
$\H^p(C_n,\H^q(U,\Gm))$ for $q>0$ are killed by some fppf cover of $U$.  Furthermore
every unit has an $n$-th root fppf-locally so that the fppf-sheafification of the presheaf $\GG_m/n$ vanishes. This implies the claim. 

It follows that the fppf-Leray spectral sequence
\begin{equation}\label{eq:leray}
    \E_2^{p,q}=\H^p_{\pl}(S,\R^q_{\pl}c_*\Gm)\Rightarrow\H^{p+q}_{\pl}(\B C_{n,S},\Gm)
\end{equation}
for $c$ takes the form given in Figure~\ref{fig:lerayss} in low degrees.
\begin{figure}[h]
    \centering
    \begin{equation*}
        \xymatrix@R=3pt{
            0&&&\\
        \mu_n(S)\ar[rrd] & \H^1_{\pl}(S, \mu_n)\ar[rrd] & &\\
        \Gm(S) & \Pic(S) & \H^2(S,\Gm) & \H^3(S,\Gm)
        }
    \end{equation*}
    \caption{The $\E_2$-page of the Leray spectral sequence~\eqref{eq:leray} computing $\H^i(\B C_n,\Gm)$.}
    \label{fig:lerayss}
\end{figure}
As $\pi^*c^* = \id$, we see that the edge homomorphisms
$\H^i(S,\Gm)\rightarrow\H^i(\B C_{n,S},\Gm)$ from the bottom line in
the Leray spectral sequence are all split injections.
In particular, the displayed differentials $d_2^{0,1}$ and $d_2^{1,1}$ are zero.

\begin{proposition}\label{prop:bcn}
    There is an isomorphism 
      $$\Gm(\B C_{n,S}) \cong \Gm(S)$$
    and short exact sequences
    \[0 \to \Pic(S) \xrightarrow{c^*} \Pic(\B C_{n,S}) \to \mu_n(S) \to 0\]
    and
    \[0\to \H^2(S,\Gm) \xrightarrow{c^*} \H^2(\B C_{n,S},\Gm) \xrightarrow{r} \H^1_{\pl}(S, \mu_n) \to 0,\]
    \[0\to \Br'(S) \xrightarrow{c^*} \Br'(\B C_{n,S}) \xrightarrow{r} \H^1_{\pl}(S, \mu_n) \to 0,\]
    \[0\to \Br(S) \xrightarrow{c^*}  \Br(\B C_{n,S}) \xrightarrow{r} \H^1_{\pl}(S, \mu_n) \to 0,\]
    which are split. The isomorphism and the short exact sequences are functorial in $\B C_{n,S}$ (i.e.\ endomorphisms of $\B C_{n,S}$ induces endomorphisms of exact sequences in a functorial manner), but the splittings are only functorial in $S$. 
\end{proposition}

\begin{proof}
    By the discussion above, the Leray spectral sequence proves everything except for the split exactness of the last
    two sequences. For the sequence involving the cohomological Brauer group, we just apply the torsion subgroup functor to the split exact sequence involving $\H^2(-, \Gm)$. By Remark \ref{rem:Gabber} we
    furthermore see that $a = \pi^*c^*a \in \H^2(S,\Gm)$ is in $\Br(S)$ if and
    only if $c^*a \in \Br(\B C_{n,S})$, implying split exactness for the last exact sequence. 
\end{proof}

Later on we will need not only the computation of the Brauer group of $\B C_n$,
but also a description of the classes coming from the inclusion $\Gm(S)/n\hookrightarrow\Br(\B
C_n)$, which is either defined via the Leray--Serre spectral sequence or using the splitting in Proposition \ref{prop:bcn}
(the proof of the following lemma will, in particular, show that these two maps differ at most by a unit). 
These classes are described via the classical cyclic algebra construction from the
previous section.

\begin{lemma}\label{lem:cyclic2}
    Let $X$ be an algebraic stack and $n$ a positive integer. Let $\sigma\in\H^1(\B
    C_{n,X},C_n)$ be the class of the universal $C_n$-torsor $X \to \B C_{n,X}$.
    Then there is an integer $k$ prime to $n$ (that only depends on $n$) such that the map 
    $$s:\H^1_{\pl}(X,\mu_n)\rightarrow\H^2(\B C_{n,X},\Gm)$$
    defined by $s(u)=k[(\sigma,u)_n]$ is a section to the map $r$ from Proposition \ref{prop:bcn}.
\end{lemma}

\begin{proof}
    It suffices to consider the universal case of $X=\B_{\pl}\mu_n$ over $\Spec\ZZ$, the stack classifying fppf $\mu_n$ torsors. 
    Note that $\B_{\pl}\mu_n$ is indeed a stack by \cite[Tag 04UR]{stacks-project} and is an algebraic stack by \cite[Tag 06DC]{stacks-project} with fppf atlas $\Spec \Z \to \B_{\pl}\mu_n$. Let
    $d:\B_{\pl}\mu_n\rightarrow\Spec\ZZ$ denote the structure map, and let
    $\R^q_{\pl}d_*\mu_n$ denote the derived functors of the push-forward in the
    fppf topos. Then, it is easy to see that $\R^0_{\pl}d_*\mu_n\iso\mu_n$ and we claim that
    $\R^1_{\pl}d_*\mu_n\iso C_n$.     
    To see the latter isomorphism, consider the
    natural transformations
    \begin{equation}\label{eq:presheaves}\Hom_{\Spec R}^{\mathrm{gp}}(\mu_{n,R},\mathds{G}_{m,R})\rightarrow\H^1(\B_{\pl}\mu_{n,R},\Gm)[n] \leftarrow\H^1_{\pl}(\B_{\pl}\mu_{n,R},\mu_n)\end{equation}
    of presheaves on affine schemes over $\ZZ$; the first map sends a homomorphism $f\colon \mu_{n,R} \to \mathds{G}_{m,R}$ to the image of the canonical class $\H^1(\B_{\pl}\mu_{n,R}, \mu_n)$ under $f_*$ and the second map is part of the Kummer sequence. The leftmost term is a sheaf and
    it is a standard fact that it is represented by the constant \'etale group
    scheme $C_n$. See~\cite{fermat}*{Section V.2.10} for example. The fppf-sheafification of the rightmost term is $\R^1_{\pl}d_*\mu_n$. To see that
    the induced map of sheaves are isomorphisms, it is sufficient to check on stalks in the
    fppf topology \cite{gabber-kelly}*{Remark 1.8, Theorem 2.3} and in particular if $R$ is a Henselian local ring with algebraically closed residue field \cite{gabber-kelly}*{Lemma 3.3}. If $R$ is such a local ring, then $\Gm(R)/n=0$ so that
    $\H^1_{\pl}(\B_{\pl}\mu_{n,R},\mu_n)\iso\H^1_{\pl}(\B_{\pl}\mu_{n,R},\Gm)[n]$.
    Using that $\Pic(R) = 0$, the Leray--Serre spectral sequence for the cover $\Spec R \to \B_{\pl}\mu_{n,R}$ shows that 
    $$\H^1_{\pl}(\B_{\pl}\mu_{n,R},\Gm) \cong \H^1_{\mathrm{group}}(\mu_{n,R},
    \mathds{G}_{m,R}) \cong \Hom^{\mathrm{gp}}_{\Spec R}(\mu_{n,R},\mathds{G}_{m,R}),$$
    where $\H^1_{\mathrm{group}}(\mu_{n,R},
    \mathds{G}_{m,R})$ is the first cohomology of the cobar complex 
    $$\Gm(S) \to \Gm(\mu_{n,R}) \to \Gm(\mu_{n,R}\times_{\Spec R} \mu_{n,R}) \to \cdots$$
    with differentials as in the usual definition of group cohomology. This shows that the morphisms in \eqref{eq:presheaves} are isomorphisms on fppf-stalks and thus that $\R^1_{\pl}d_*\mu_n\iso C_n$.

    Now, the fppf-Leray spectral sequence for
    $d:\B_{\pl}\mu_n\rightarrow\Spec\ZZ$ yields an exact sequence
    \begin{equation}\label{eq:munsplitex} 0\rightarrow\H^1_{\pl}(\Spec\ZZ,\mu_n)\rightarrow\H^1_{\pl}(\B_{\pl}\mu_n,\mu_n)\rightarrow\H^0(\Spec\ZZ,C_n)\rightarrow
    0.\end{equation}
    The right hand term is isomorphic to $\ZZ/n$, and the sequence is
    split by applying the pullback map along
    $\Spec\ZZ\rightarrow\B_{\pl}\mu_n$. We denote by $\tau\in \H^1_{\pl}(\B_{\pl}\mu_n,\mu_n)$ the class of the universal
    $\mu_n$-torsor over $\B_{\pl}\mu_n$. This is (exactly) of order $n$ and pulls back to zero on $\Spec \ZZ$. 

    Consider $c\colon \B C_{n,\B_{\pl}\mu_n}\rightarrow\B_{\pl}\mu_n$ and the class $\alpha=[(\sigma,c^*\tau)_n]$. The
    class of $\alpha$ has order (exactly) $n$ as there are cyclic
    algebras of order $n$ over fields, for example
    by~\cite{gille-szamuely}*{Lemma~5.5.3}. As
    $\pi^*\alpha=0$ for $\pi\colon \B_{\pl}\mu_n \to \B C_{n,\B_{\pl}\mu_n}$ the projection, it follows from Proposition \ref{prop:bcn} that $r(\alpha)$ in
    $\H^1_{\pl}(\B_{\pl}\mu_n,\mu_n)$ has order $n$ as well. On the other hand, $r(\alpha)$ pulls back to
    zero over $\Spec\ZZ$ so it is a non-zero multiple of $\tau$ (using the split-exact sequence \eqref{eq:munsplitex}). Thus, $r(\alpha)=m\tau$ for some $m$ prime to $n$.
    This completes the proof if we set $k$ to be a number such that $km \equiv 1 \mod n$.
\end{proof}

 \begin{corollary}\label{cor:cyclicgeneral}
     Suppose that $\chi\colon X\to Y$ is a $C_n$-torsor for some positive
     integer $n$.
     Let $u\in\Gm(Y)/n$ be the class of a unit, and write $\alpha_u$ for the corresponding
     class in $\Br'(Y)$ (defined via the Leray--Serre spectral sequence). Then we
     have $\alpha_u=k[(\chi,u)]$ in $\Br'(Y)$, where $k$ is some number prime to $n$ which only depends on $n$. 
 \end{corollary}

We do not know the value of $k$ in the corollary. Perhaps it is always $\pm 1$, as is
the case in similar computations, such as the result of
Lichtenbaum (see~\cite{gille-szamuely}*{Theorem~5.4.10}), which computes the
exact value of the map $\Pic(X_{\overline{k}})^G\iso\ZZ\rightarrow\Br(k)$ when $X$
is a Severi--Brauer variety of a field with Galois group $G$, or the computation
of~\cite{gille-queguiner-mathieu} of the sign of the Rost invariant.

\begin{proposition}\label{prop:cyclicpresentationallstacks}
    Let $\Xscr$ be an algebraic stack and suppose that
    $\chi\in\H^1(\Xscr,C_n)$ and $u\in\H^1_{\pl}(\Xscr,\mu_n)$ are fixed classes. In the notation above,
    we have $[(\chi,u)_n]=[\Ascr_{\chi,u}]$ in $\Br'(\Xscr)$.
\end{proposition}

\begin{proof}
    Both $[\Ascr_{\chi,u}]$ and $[(\chi,u)_n]$ define classes in $\H^2_{\pl}(\B
    C_n\times\B_{\pl}\mu_n,\Gm)$. As at the end of the proof of
    Proposition~3.3, we see that
    $[\Ascr_{\chi,u}]=k[(\chi,u)_n]$ for some $k$ prime to $n$.
    We saw in Proposition~\ref{prop:cyclicpresentation} that they agree when pulled
    back to regular noetherian schemes. The result follows.
\end{proof}

\section{A presentation of the moduli stack of elliptic curves}\label{sec:presentation}

We will compute $\Br(\Mscr)$ using that it injects into
$\Br(\Mscr_{\ZZ[\tfrac12]})$ by Proposition~\ref{prop:BrOmnibus}(iv) and using a specific presentation of
$\Mscr_{\ZZf{2}}$, which we now describe. This presentation is standard and we
claim no originality in our presentation of it. For references,
see~\cite{deligne-rapoport} or~\cite{katz-mazur}, 

\begin{definition}
    A \df{full level $2$ structure} on an elliptic curve $E$ over a base scheme $S$
    is a fixed isomorphism $(\ZZ/2)^2_S\rightarrow E[2]$, where $(\ZZ/2)^2_S$
    denotes the constant group scheme on $(\ZZ/2)^2$ over $S$ and $E[2]$ is the subgroupscheme
    of order $2$ points in $E$. If there exists an isomorphism
    $(\ZZ/2)^2_S\cong E[2]$, an equivalent way of specifying a level $2$
    structure is to order the points of exact order $2$ in $E(S)$ (over each
    connected component of $S$). 
\end{definition}

\begin{remark}
    These full level $2$ structures are sometimes called \emph{naive} to
    distinguish them from the level structures considered by Drinfeld, which
    allow one to extend $\Mscr(2)$ to a stack supported over all of $\Spec\ZZ$. We will not need this
    generalization in this paper. It is the subject of~\cite{katz-mazur}.
\end{remark}

The moduli stack $\Mscr(2)$ of elliptic curves with fixed level $2$ structures
is a regular noetherian Deligne-Mumford stack. Moreover, since the existence of
a full level $2$ structure implies that $2$ is invertible in $S$
(by~\cite{katz-mazur}*{Corollary~2.3.2} for example), the functor
$\Mscr(2)\rightarrow\Mscr$ which forgets the level structure factors through
$\Mscr_{\ZZf{2}}$. This map is clearly equivariant for the right $S_3$-action
on $\Mscr(2)$ that permutes the non-zero $2$-torsion points and the trivial
$S_3$-action on $\Mscr$. Note that in general, being $G$-equivariant for a map
$f\colon \mathcal{X} \to \mathcal{Y}$ of stacks with $G$-action is extra
structure: for every $g\in G$ one has to provide compatible $2$-morphism
$\sigma_g\colon gf \to fg$ (see \cite{romagny} for details). In the case of
$\Mscr(2) \to \Mscr_{\Z[\frac12]}$ though the equivariance is strict in the
sense that all $\sigma_g$ are the identity $2$-morphisms of
$\Mscr(2)\rightarrow\Mscr_{\ZZf{2}}$. 

We have the following well-known statement; to fix ideas, we will provide a proof.

\begin{lemma}\label{lem:S3torsor}
    The map $\Mscr(2)\rightarrow\Mscr_{\ZZf{2}}$ is an $S_3$-Galois cover.
\end{lemma}

\begin{proof}
    It is enough to show that for every affine scheme $\Spec R$ over
    $\Spec\ZZf{2}$ and every elliptic curve $E$ over $\Spec R$, we can find a
    full level $2$ structure \'etale locally. Indeed, if there is one full level $2$ structure on $E$, the map
    $$ (S_3)_{\Spec R}=S_3 \times \Spec R \to \Spec R \times_{\MM_{\ZZf{2}}} \MM(2)$$
    from the constant group scheme on $S_3$
    is an isomorphism since we get every other full level $2$ structure on $E$ by permuting the non-zero $2$-torsion points. 

    The elliptic curve $E$ defines an $R$-point $E:\Spec R\rightarrow\Mscr_{\ZZf{2}}$ of the
    moduli stack of elliptic curves. Zariski locally we can assume the pullback $E^*\lambda$ of the Hodge
    bundle to be trivial, in which case there exists a
    nowhere vanishing invariant differential $\omega$.
    By~\cite{katz-mazur}*{Section~2.2}, we can then write $E$ in Weierstrass form over $\Spec R$,
    which after a coordinate change takes the form
    $$y^2 = x^3 + b_2x^2+b_4x+b_6.$$
    As a point $(x,y)$ on $E$ is $2$-torsion if and only if $y=0$, we have a
    full level $2$ structure after adjoining the three roots $e_1,e_2$ and $e_3$ of
    $x^3 + b_2x^2+b_4x+b_6$ to $R$. This defines an \'etale extension as the
    discriminant of this cubic polynomial does not vanish (because $E$ is
    smooth). 
\end{proof}

\begin{definition}
    A \df{Legendre curve} with parameter $t$ over $S$ is an elliptic curve
    $E_t$ with Weierstrass
    equation $$y^2=x(x-1)(x-t).$$ 
    As the discriminant of this equation is $16t^2(t-1)^2$, such an equation defines an elliptic (and hence Legendre) curve
    if and only if $2, t$ and $t-1$ are invertible on $S$. 
\end{definition}

The points $(0,0)$, $(1,0)$, and $(t,0)$ define three non-zero $2$-torsion
points on $E_t$. Taking them in this order fixes a full level $2$ structure on
$E$. This defines a morphism $$\pi\colon X\rightarrow\Mscr(2),$$ where $X$ is
the parameter space of Legendre curves. In fact, $X$ is an affine scheme, given as
$$X=\Spec\ZZ\left[\tfrac12,t^{\pm 1},(t-1)^{-1}\right]=\AA^1_{\ZZ[\frac12]}-\{0,1\}.$$
We will use $X$ throughout this paper to refer specifically to this moduli
space of Legendre curves. In particular, $X$ is naturally defined over
$\ZZf{2}$. In general, given a scheme $S$, we let $X_S=\AA^1_S-\{0,1\}$.
Note that this is a slight abuse of notation as we do not assume that $2$ is
invertible on $S$.

We equip the map $\pi\colon X \to \Mscr(2)$ with the structure of a $C_2$-equivariant map with the trivial $C_2$-action on $X$ and $\Mscr(2)$ by choosing $\sigma_g\colon g\pi \to
    \pi g$ to be $[-1]$ (i.e.\ multiplication by $-1$ on the universal
    elliptic curve) for $g \in C_2$ the non-trivial element. Note that $[-1]$
    fixes the level $2$ structure and so indeed defines a natural automorphism
    of $\id_{\Mscr(2)}$. The structure of a $C_2$-equivariant map on $\pi$ induces a map $[X/C_2] = \B C_{2,X} \to \Mscr(2)$. 

\begin{proposition}\label{prop:Xtorsor}
    The $C_2$-equivariant map $\pi\colon X\rightarrow\Mscr(2)$ is a $C_2$-torsor. Thus, the map $\B C_{2,X}\to\Mscr(2)$ is an equivalence. 
\end{proposition}

\begin{proof}
    First we will show that an elliptic curve $E\rightarrow\Spec R$ with full level $2$ structure can \'etale locally be brought into Legendre form. Our proof will be
    along the lines of \cite{silverman}*{Proposition III.1.7}, but we have to take
    a little bit more care.  
    
    As in the proof of Lemma \ref{lem:S3torsor}, Zariski locally over $\Spec
    R$, we can write $E$ in the form
    $$y^2 = x^2 + b_2x^2+b_4x+b_6$$
    and the full level $2$ structure allows us to factor the right hand side as
    $$(x-e_1)(x - e_2)(x-e_3),$$
    where $(e_1,0)$, $(e_2,0)$ and $(e_3,0)$ are the nonzero $2$-torsion points. We
    set $p = e_2-e_1$ and $q= e_3-e_1$. By a linear coordinate change, we get $y^2
    = x(x-p)(x-q)$. 

    Since the equation $y^2=x(x-p)(x-q)$ defines an elliptic curve, $p$, $q$ and $p-q$ are nowhere
    vanishing. Thus, the extension $R \to R[\sqrt p]$ is \'etale so that we can
    (and will) assume \'etale locally to have a (chosen) square root
    $\sqrt{p}$. Now, $E$ is isomorphic to
    $y^2 = x(x-1)(x-t)$ for $t=\frac{q}p$, where the isomorphism is given by $x
    = px'$ and $y = p^{3/2}y'$. Thus our original $E$ is indeed \'etale locally (on the base) isomorphic to a Legendre curve as an elliptic curve with level $2$ structure. It is moreover an elementary check with coordinate
    transformations that there is at most one choice of $t \in R$ such that the Legendre curve $E_t$ with paramter $t$ is isomorphic to $E$ in $\Mscr(2)$. 
    
    Now assume that our elliptic curve $E$ over $R$ is in Legendre form and assume further that $\Spec R$ is connected. By definition, for a commutative $R$-algebra $R'$, an element of $(X\times_{\MM(2)} \Spec
    R)(R')$ consists of a Legendre curve $E_t$ together with an isomorphism of
    $E_t$ to $E_{R'}$ in $\Mscr(2)$. By assumption this set is
    non-empty and it is indeed a torsor under the group of automorphisms of $E_{R'}$ in $\Mscr(2)$. 
    By \cite[Corollary 2.7.2]{katz-mazur}, the only non-trivial automorphism of $E_{R'}$ with level $2$
    structure is $[-1]$. Thus, the $C_2$-action exactly interchanges the two elements of $(X\times_{\MM(2)} \Spec R)(R')$ and we obtain a $C_2$-equivariant equivalence
    $$C_2\times \Spec R \simeq X\times_{\MM(2)} \Spec R.$$
    As every $E$ with full level $2$ structure satisfies \'etale locally our assumptions, this implies that $X \to \MM(2)$ is a $C_2$-torsor. 
    
    By the general fact that for a $G$-torsor $\mathcal{X} \to \mathcal{Y}$,
    the induced map $[\mathcal{X}/G] \to \mathcal{Y}$ is an equivalence, we
    obtain in our case the equivalence $\B C_{2,X} \simeq \Mscr(2)$. 
\end{proof}

\begin{corollary}\label{cor:coarse}
    The map $c\colon \Mscr(2)\rightarrow X$ sending $y^2=(x-e_1)(x-e_2)(x-e_3)$ to
    $y^2=x(x-1)(x-\frac{e_3-e_1}{e_2-e_1})$ exhibits $X$ as the coarse moduli
    space of $\Mscr(2)$.
\end{corollary}
\begin{proof}
 The set of maps from $\Mscr(2) \simeq \B C_{2,X}$ to $X$ is in bijection with $C_2$-equivariant maps $X \to X$. Thus, a map $\Mscr(2) \to X$ exhibits $X$ as the coarse moduli space if and only if the precomposition with $\pi$ is the identity. This is clearly the case for $c$. 
\end{proof}

It follows that the right $S_3$-action on $\Mscr(2)$ induces a right $S_3$ action on $X$. We can describe this
explicitly as follows. Consider the
generators $\sigma = (1\,3\,2)$ and $\tau = (2\,3)$ of
$\GL_2(\ZZ/2)\iso S_3$, of orders $3$ and $2$, respectively. Then,
\begin{align*}
    \sigma(t)&=\frac{t-1}{t},\\
    \tau(t)&=\frac{1}{t}.
\end{align*}
By a simple computation, the map $c \colon \MM(2) \to X$ defined above is
strictly $S_3$-equivariant. 

In contrast, the map $\pi\colon X \to \MM(2)$ described above is \emph{not} $S_3$-equivariant, as one
notes for example by checking that the elliptic curves $y^2 = x(x-1)(x-t)$
and $y^2=x(x-1)(x-\frac1t)$ are generally not isomorphic.
To actually explain the correct $S_3$-action on $\B C_{2,X}$, we have
to fix some notation.

Consider again an elliptic curve $E$ given by $y^2 =
(x-e_1)(x-e_2)(x-e_3)$. Set again $p= e_2-e_1$ and $q = e_3-e_1$ so that we can
write $E$ as
$$y^2 = x(x-p)(x-q).$$ 
The only possible coordinate changes fixing the form of this equation are the
transformations $y \mapsto u^3y$ and $x\mapsto
u^2x$; such a coordinate change results in multiplying the standard invariant
differential $\omega = \frac{-dx}{2y}$ by $u^{-1}$ and
sending $p$ to $u^2p$ and $q$ to $u^2q$. Thus, $p\omega^{\otimes 2}$ and
$q\omega^{\otimes 2}$ define canonical sections of $\lambda^{\tensor 2}$ on
$\Mscr(2)$, not dependent on any choice of Weierstrass form. Note that these
sections are nowhere vanishing. We can consider the
$C_2$-torsor $\Mscr(2)(\sqrt{p}) \to \Mscr(2)$ defined as the cyclic cover
$\mathbf{Spec}_{\Mscr(2)}\left(\bigoplus_{i\in\ZZ} \lambda^{\tensor
i}/(1-p)\right)$. \'Etale locally
on
some $\Spec R$, we can trivialize $\lambda$ so that $p$ becomes an element of
$R$ and the $C_2$-torsor becomes $\Spec R[\sqrt{p}] \to \Spec R$. The
$C_2$-torsor $\Mscr(2)(\sqrt{p}) \to \Mscr(2)$ is equivalent to $X \to
\Mscr(2)$. Indeed, we have shown in the proof of Proposition \ref{prop:Xtorsor}
that the latter has a section as soon as we have a chosen square root of $p$.

As $g^*\lambda$ for $g\in S_3$ on $\Mscr(2)$ is canonically isomorphic to
$\lambda$ (as this is pulled back from $\Mscr$), we have an action of $S_3$ on
$\H^0(\Mscr(2), \lambda^{\tensor *})$. Consider the section 
$$\frac{g(p)}p\in \H^0(\Mscr(2), \OO_{\Mscr(2)}) \cong\H^0(X, \OO_X),$$
which can for $E$ as above be written as
$\frac{e_{g(2)}-e_{g(1)}}{e_2-e_1}$. For example, we have $\frac{g(p)}p =
\frac{q}{p}$ for $g = \tau$, which equals $t$ on $X$. For a scheme $S$ with a map $f\colon S \to X$ or $f\colon S \to \MM(2)$, we denote 
the torsor adjoining the square root of $f^*\frac{g(p)}p$ by $T_{f,g} \to S$.

For the next lemma, we recall that an object in $\B C_{2,X}(S)$
corresponds to a $C_2$-torsor $T\to S$ and a $C_2$-equivariant map $T\to
X$, where $X$ has the trivial $C_2$-action. Equivalently, an object can be described as a $C_2$-torsor $T\rightarrow S$ with a map $f\colon S\to X$. 
Let $S_3$ act on $\B C_{2,X}$ in the following way: $g\in S_3$ acts (from the right) on $(T,f) \in \B C_{2,X}(S)$ by 
setting $g(T)$ to be $(T\times_S T_{f,g})/C_2$ and the map $g(f)$ to be the composition
$S \xrightarrow{f} X \xrightarrow{g} X$.

\begin{lemma}\label{lem:S3action}
    The natural map $\MM(2)\rightarrow\B C_{2,X}$
    induces an $S_3$-equivariant equivalence $\B C_{2,X} \simeq \MM(2)$.
\end{lemma}

\begin{proof}
    As noted above, the map $\MM(2) \to \B C_{2,X}$ classifying the torsor
    $\MM(2)(\sqrt{p}) \to \MM(2)$ is an equivalence by Proposition
    \ref{prop:Xtorsor} (as this torsor is equivalent to $X \to \MM(2)$). We
    have only to check the $S_3$-equivariance of this map.

    Given an $f\colon S \to \MM(2)$, the corresponding object in $\B C_{2,X}$
    is the torsor $S(\sqrt{f^*p}) \to S$ together with $S \to \MM(2) \to X$.
    The composition $gf$ for $g\in S_3$ corresponds to the torsor
    $S(\sqrt{(gf)^*p}) \to S$ together with $S \to \MM(2) \to X \xrightarrow{g}
    X$ as $\MM(2) \to X$ is $S_3$-equivariant. As $(gf)^*p = f^*(g(p))$, we
    have $(gf)^*p = f^*p \cdot f^*\left(\frac{g(p)}{p}\right)$. Thus, we have a
    natural isomorphism $(S(\sqrt{p})\times_S T_{f, g})/C_2 \xrightarrow{\cong}
    S(\sqrt{(gf)^*p})$. One can check that these isomorphisms are compatible
    (similarly to~\cite[Definition 2.1]{romagny}, although we do not have a strict
    $S_3$-action on $\B C_{2,X}$) so that one actually gets the structure of an
    $S_3$-equivariant map.
\end{proof}

Of particular import will be the action of $S_3$ on the units of $X$.
Let $\rho$ be the tautological permutation representation of $S_3$ on $\ZZ^{\oplus 3}$ and let $\tilde{\rho}$ be the kernel of the morphism 
$$ \rho \cong \ind_{C_2}^{S_3}\ZZ \to \ZZ$$
to the trivial representation, the adjoint to the identity. 

\begin{lemma}\label{lem:units}
    For any connected normal noetherian scheme $S$ over $\ZZ[\frac12]$, there is an $S_3$-equivariant exact sequence
    $$0\rightarrow\Gm(S)\rightarrow\Gm(X_S)\rightarrow \tilde{\rho}\rightarrow 0,$$
    where $S_3$ acts on $\Gm(S)$ trivially and $\tilde{\rho}$ is additively generated by the images of $t$ and $t-1$. This exact sequence is non-equivariantly split. 
\end{lemma}

\begin{proof}
    Denote by $\pi\colon X_S \to S$ the structure map. We have a map $f\colon
    \ZZ^2\oplus \Gm \to \pi_*\mathds{G}_{m,X_S}$ of sheaves on $S$, where $f$
    takes the two $\ZZ$-summands to $t$ and $(t-1)$, respectively. We claim
    that this map is an isomorphism. It is enough to check this on affine
    connected opens $\Spec R$, where it follows from $R$ being an integral
    domain (as it is normal). The non-equivariant statement follows.

    Moreover, the action of $S_3$ on $\Gm(S)$ is
    trivial by definition. Set $\sigma = (1\,3\,2)$ and $\tau
    = (2\,3)$. If we choose the basis vectors
    $\begin{pmatrix}0\\1\\-1\end{pmatrix}$ and $\begin{pmatrix}1\\0\\-1\end{pmatrix}$ 	
    for $\tilde{\rho}$, we obtain exactly the same $S_3$-representation as on $\Gm(X_S)/\Gm(S) \cong \ZZ\{t,t-1\}$, where the latter denotes the free $\Z$-module on $t$ and $t-1$ with elements thought of as $t^k(t-1)^l$. 
\end{proof}

\section{Beginning of the computation}\label{sec:beginning}

Let $S$ be a connected regular noetherian scheme over $\ZZ[\tfrac12]$, let $\Mscr_S$ be the moduli stack of
elliptic curves over $S$, and let $\Mscr(2)_S$ be the moduli stack of elliptic
curves with full level $2$ structure over $S$.
The Leray--Serre spectral sequence for $\Mscr(2)_S \to \Mscr_S$ 
 takes the form 
\begin{equation}\label{eq:lsss}
    \E_2^{p,q}:\H^p(S_3,\H^q(\Mscr(2)_S,\Gm))\Rightarrow\H^{p+q}(\Mscr_S,\Gm),
\end{equation}
with differentials $d_r$ of bidegree $(r,1-r)$. In this section, we will
collect the basic tools to compute the $E_2$-term. 
We start with two brief remarks about the cohomology of $S_3$.

 

\begin{lemma}\label{lem:htriv}
    Let $M$ be a trivial $S_3$-module. Then,
    \begin{align*}
        \H^1(S_3,M) &\iso M[2],\\
        \H^2(S_3,M) &\iso M/2,\\
        \H^3(S_3,M) &\iso M[6],\\
        \H^4(S_3,M) &\iso M/6.
    \end{align*}
\end{lemma}

\begin{proof}
    We use the Lyndon--Hochschild--Serre spectral sequence for 
    \[1\to \ZZ/3 \to S_3 \to \ZZ/2 \to 1.\] 
    A reference
    is~\cite{weibel}*{Example 6.7.10}. On the $\E_2$-page, $\E_2^{pq}=0$ whenever
    $p>0$ and $q>0$ because the cohomology of $\ZZ/3$ is $3$-torsion.
    Moreover, $\ZZ/2$ acts on $\H^q(\ZZ/3,M)$ by multiplication by
    $-1$ for $q\equiv 1,2\,\mathrm{mod}\, 4$ and by $1$ for $q\equiv
    0,3\,\mathrm{mod}\, 4$.
\end{proof}

%

The next lemma is about the cohomology of the reduced regular representation
$\tilde{\rho}$ of $S_3$ introduced in Section~\ref{sec:presentation}.

\begin{lemma}\label{lem:hrho}
    Let $M$ be an abelian group, and let $\tilde{\rho}\otimes M$ be an
    $S_3$-module through the action on $\tilde{\rho}$. Then,
    \begin{align*}
        \H^0(S_3,\tilde{\rho}\otimes M) &\iso M[3],\\
        \H^1(S_3,\tilde{\rho}\otimes M) &\iso M/3,\\
        \H^2(S_3,\tilde{\rho}\otimes M) &\iso 0\\
        \H^3(S_3,\tilde{\rho}\otimes M) &\iso 0.
    \end{align*}
\end{lemma}

\begin{proof}
    There is a short exact sequence of $S_3$-modules 
    $$0\rightarrow \tilde{\rho}\otimes M\rightarrow\rho\otimes M\rightarrow M\rightarrow 0.$$
     In the associated long exact sequence in cohomology, note that
    $\H^i(S_3, \rho\otimes M) \cong\H^i(C_2,M)$ by Shapiro's lemma, as
    $\rho\otimes M\iso \ind_{C_2}^{S_3}M$. The map 
    $$\H^i(S_3,\rho\otimes M) \cong\H^i(C_2,M) \to\H^i(S_3,M)$$
    is the transfer. We obtain short exact sequences
    $$ 0 \to \coker\, \tr_{C_2}^{S_3}(\H^{i-1}) \to\H^i(S_3,\tilde{\rho}\otimes
    M) \to \ker \tr_{C_2}^{S_3}(\H^i) \to 0.$$
    Because $C_2 \to S_3$ has a retraction, the restriction map $\H^i(S_3,M)
    \to\H^i(C_2,M)$ is the projection to a direct summand. The transfer equals
    $3$ times the inclusion of this summand as can easily be deduced from the
    equation $\mathrm{tr}_{C_2}^{S_3} \res_{C_2}^{S_3} = 3$. Thus, the transfer
    is multiplication by $3$ on $\H^0$, an isomorphism
    on $\H^1$ and $\H^2$ and the inclusion $M[2] \to M[6]$ on $\H^3$. The lemma follows.
\end{proof}

These computations allow us to compute the $\E_2$-term of the Leray--Serre spectral sequence
\[\E_2^{p,q}:\H^p(S_3,\H^q(\Mscr(2)_S,\Gm))\Rightarrow\H^{p+q}(\Mscr_S,\Gm)\]
in a range. Using the results of the last two sections, we can analyze $\H^q(\Mscr(2)_S,\Gm)$ in terms of $\H^q(X_S,\Gm)$. Especially
Proposition~\ref{prop:bcn} turns out to be useful as the short exact sequences in it are $S_3$-equivariant by naturality. Using additionally Lemma \ref{lem:units} for the first one, we obtain the $S_3$-equivariant extensions
\begin{align}\label{eq:GmXS} 0\rightarrow\Gm(S)\rightarrow\Gm(\MM(2)_S) \cong \Gm(X_S)\rightarrow
    \tilde{\rho}\rightarrow 0,
\end{align}
\begin{align}\label{eq:PicXS}
    0 \to \Pic(X_S) \cong \Pic(S) \to \Pic(\MM(2)_S) \to \mu_2(S) \to 0,
\end{align}
and 
\begin{align}\label{eq:BrXS}
    0 \to \Br'(X_S) \to \Br'(\MM(2)_S) \to\H^1(X_S, \mu_2) \to 0. 
\end{align}
The only point needing justification is that the pullback map $\Pic(S) \to \Pic(X_S)$ is an isomorphism. It is injective because $X_S$
has an $S$-point. It is surjective as it factors through the isomorphism $\Pic(S) \to \Pic(\AA^1_S)$
and since $j^*\colon \Pic(\AA^1_S) \to \Pic(X_S)$
is surjective, where $j$ denotes the inclusion $X_S\subseteq\AA^1_S$. Indeed, given a line bundle $\Lscr$ on $X_S$, we take a coherent
subsheaf $\Fscr$ of $j_*\Lscr$ with $j^*\Fscr \cong\Lscr$. The
double dual of $\Fscr$ is a reflexive sheaf $\Lscr'$ with $j^*\Lscr'$ still isomorphic
to $\Lscr$. By \cite[Prop 1.9]{hartshorne-reflexive}, $\Lscr'$ is a line bundle.

The sequence \eqref{eq:PicXS} is $S_3$-equivariantly split and thus consists only
of $S_3$-modules with the trivial action. Indeed, the morphism $S \to \Spec
\Z[\frac12]$ induces by pullback a morphism from the exact sequence
$$0 \to 0 \to \Pic(\MM(2)) \to \mu_2(\Z[\tfrac12]) \to 0,$$
where the splitting is clearly $S_3$-equivariant. As $\mu_2(\Z[\frac12]) \to
\mu_2(S)$ is an isomorphism for $S$ connected, the result follows. These observations allow us to compute
the $q=0,1$ lines of the Leray--Serre spectral sequence~\eqref{eq:lsss}.

\begin{lemma}\label{lem:twolines}
    If $S$ is a connected regular noetherian scheme over $\ZZf{2}$, then there are
    natural extensions
    \begin{equation*}
        0\rightarrow\H^p(S_3,\Gm(S))\rightarrow\H^p(S_3,\Gm(\Mscr(2)_S))\rightarrow\H^p(S_3,\tilde{\rho})\rightarrow 0
    \end{equation*}
    for $0\leq p\leq 3$, and natural isomorphisms
    \begin{align*}
        \H^p(S_3, \H^1(\Mscr(2)_S,\Gm)) &\cong \H^p(S_3, \Pic(S)) \oplus \H^p(S_3, \mu_2(S))
    \end{align*}
    for all $p\geq 0$.
\end{lemma}

\begin{proof}
    The first exact sequence
    follows from Lemma~\ref{lem:htriv} and Lemma~\ref{lem:hrho} using that
    $\H^1(S_3,\tilde{\rho})$ is $3$-torsion and $\H^2(S_3,\Gm(S))$ is
    $2$-torsion. The direct sum decomposition follows from the fact
    that~\eqref{eq:PicXS} is $S_3$-equivariantly split. 
\end{proof}

The only necessary remaining group we need to understand for our computations
is $\H^2(\MM(2)_S,\Gm)^{S_3}$, which we analyze using
the short exact sequence \eqref{eq:BrXS}.

\begin{lemma}\label{lem:mu}
    If $S$ is a regular noetherian scheme over $\ZZf{2}$, then there is a
    canonical isomorphism
    $\H^1(S,\mu_2)\iso\H^1(X_S,\mu_2)^{S_3}$.
\end{lemma}

\begin{proof}
    Using the $S_3$-equivariant short exact sequence
    $$0\rightarrow\Gm(X_S)/2\rightarrow\H^1(X_S,\mu_2)\rightarrow\Pic(X_S)[2]\rightarrow 0,$$
    we get a long exact sequence
    \[0 \to \left(\Gm(X_S)/2\right)^{S_3} \to\H^1(X_S, \mu_2)^{S_3} \to
        \Pic(X_S)[2]^{S_3} \to\H^1(S_3, \Gm(X_S)/2) \to \cdots \]
    As the canonical map $X_S \to S$ is $S_3$-equivariant, we obtain a map into this from the exact sequence
    \[0 \to \Gm(S)/2 \to\H^1(S, \mu_2) \to \Pic(S)[2] \to 0. \]
    As the maps $\Gm(S)/2 \to \left(\Gm(X_S)/2\right)^{S_3}$ and $\Pic(S)[2]
    \to \Pic(X_S)[2]$ are isomorphisms (using the exact sequence \eqref{eq:GmXS} and Lemma \ref{lem:hrho}),
    the five lemma implies that $\H^1(S, \mu_2) \to\H^1(X_S, \mu_2)^{S_3}$ is an
    isomorphism as well.
\end{proof}

From~\eqref{eq:BrXS}, we obtain a long exact sequence
\begin{align}\label{eq:Brseq}0 \to \Br'(X_S)^{S_3} \to \Br'(\MM(2))^{S_3} \to \H^1(S,\mu_2) \to\H^1(S_3,
    \Br'(X_S)) \to \cdots\end{align}
    

\begin{lemma}\label{lem:ptorsion}
    Let $S$ be a regular noetherian scheme over $\Spec\ZZ[\frac{1}{p}]$ for
    some prime $p$, and let
    $X_S=\AA^1_S-\{0,1\}$ as before. There is a non-canonically split exact sequence
    \begin{equation*}
        0\rightarrow{_p\Br'(S)}\rightarrow{_p\Br'(X_S)}\rightarrow{_p}\H^3_{\{0,1\}}(\AA^1_S,\Gm)\rightarrow 0.
    \end{equation*}
\end{lemma}

\begin{proof}
    By Proposition \ref{prop:puritycalc} we have an exact sequence 
    $$0\rightarrow {_p}\Br'(\AA^1_S)\rightarrow {_p}\Br'(X_S)\rightarrow {_p}\H^3_{\{0,1\}}(\AA^1_S,\Gm)\rightarrow {_p}\H^3(\AA^1_S,\Gm)\rightarrow {_p}\H^3(X_S,\Gm).$$   
    Because $p$ is invertible on $S$, Proposition \ref{prop:BrOmnibus} implies that
    ${_p\H^i(S,\Gm)}\iso{_p\H^i(\AA^1_S,\Gm)}$ for all $i\geq 0$. But, since $X_S$ has an $S$-point,
    it follows that ${_p\H^i(\AA^1_S,\Gm)}\cong {_p\H^i(S,\Gm)}
    \rightarrow{_p\H^i(X_S,\Gm)}$ is split injective for all $i$.
\end{proof}

\begin{lemma}\label{lem:supportsrho}
    For any prime $p$ and any regular noetherian scheme over
    $\Spec\ZZ[\tfrac1p]$, there is a canonical isomorphism
    $$\H^q_{\{0,1\}}(\AA^1_S,\Gm)\iso\H^q_{\{0\}}(\AA^1_S,\Gm)\oplus\H^q_{\{1\}}(\AA^1_S,\Gm).$$
    The action of $S_3$ on ${_p\Br'(X_S)}/{_p\Br'(S)}$ is isomorphic to   
    $$\tilde{\rho}\otimes{_p\H^3_{\{0\}}(\AA^1_S,\Gm)}\cong \tilde{\rho}\otimes\H^1(S, \QQ_p/\ZZ_p).$$
\end{lemma}

\begin{proof}
    Given any \'etale sheaf $\Fscr$, there is a canonical isomorphism
    $$\H^0_{\{0,1\}}(\AA^1_S,\Fscr)\iso\H^0_{\{0\}}(\AA^1_S,\Fscr)\oplus\H^0_{\{1\}}(\AA^1_S,\Fscr),$$
    as one sees by an easy diagram chase. By deriving this isomorphism, the first part of the
    lemma follows.

    To prove the second statement, we compare the sequence of
    Lemma~\ref{lem:ptorsion} with the long exact sequence for \'etale
    cohomology with supports coming from the open inclusion
    $X_S\subseteq\PP^1_S$. Using the natural map of long exact sequences, we
    obtain a commutative diagram
    \begin{equation*}
        \xymatrix{
            0\ar[r] &   {_p\Br'(\PP^1_S)}\ar[r]\ar[d]^\iso   &
            {_p\Br'(X_S)}\ar[r]\ar[d]^= &
            {_p\H^3_{\{0,1,\infty\}}(\PP^1_S,\Gm)}\ar[d]&\\
            0\ar[r] &   {_p\Br'(\AA^1_S)}\ar[r]         &   {_p\Br'(X_S)}\ar[r]       &
            {_p\H^3_{\{0,1\}}(\AA^1_S,\Gm)}\ar[r]   &   0
        }
    \end{equation*}
    with exact rows,
    where the left-hand vertical map is an isomorphism because it is injective
    (by Proposition~\ref{prop:BrOmnibus}(iv)),
    $_p\Br'(S)\rightarrow{_p\Br'(\AA^1_S)}$ is an isomorphism, and there is an
    $S$-point of $\AA^1_S\subseteq\PP^1_S$.

    Now, by Proposition~\ref{prop:puritycalc},
    $${_p\H^3_{\{0,1,\infty\}}(\PP^1_S,\Gm)}\iso\bigoplus_{\{0,1,\infty\}}\H^1(S,\QQ_p/\ZZ_p)$$
    and
    $${_p\H^3_{\{0,1\}}(\AA^1_S,\Gm)}\iso\bigoplus_{\{0,1\}}\H^1(S,\QQ_p/\ZZ_p).$$
    With this description, the right-hand vertical map above is the natural
    projection away from the factor of $\H^1(S,\QQ_p/\ZZ_p)$ corresponding to
    $\infty$. Let $\chi_0$ and $\chi_1$ be $p$-primary characters of $S$, i.e.,
    elements of $\H^1(S,\QQ_p/\ZZ_p)$. Then, as $\chi_0,\chi_1$ vary, the Azumaya algebras
    $(\chi_0,t)\otimes(\chi_1,t-1)$ give elements of $\Br(X_S)$ whose
    ramification classes $(\chi_0,\chi_1)$ span $_p\H^3_{\{0,1\}}(\AA^1_S,\Gm)$.
    The ramification of such a class computed in $\H^3_{\{0,1,\infty\}}(\PP^1_S,\Gm)$
    is $(\chi_0,\chi_1,-\chi_0-\chi_1)$. This follows from
    Proposition~\ref{prop:ramification}, the fact that
    $\mathrm{ram}_{(\pi)}(\chi,\pi^{-1})=-\mathrm{ram}_{(\pi)}(\chi,\pi)$ in
    the notation of that proposition, and the fact that both $t^{-1}$ and
    $(t-1)^{-1}$ are uniformizing parameters for the divisor at $\infty$ of
    $\PP^1_S$. It follows that the image
    of $_p\Br'(X_S)$ inside
    $_p\H^3_{\{0,1,\infty\}}(\PP^1_S,\Gm)\iso\bigoplus_{\{0,1,\infty\}}\H^1(S,\QQ_p/\ZZ_p)$
    can be identified with $\tilde{\rho}\otimes\H^1(S,\QQ_p/\ZZ_p)$.
\end{proof}
    
We will analyze the implications for $p$-primary torsion for $p>2$ in the next
three sections. For the rest of this section, we will begin the study of the $2$-primary
torsion of $\Br'(\Mscr_S)$ where $S$ is a $\ZZf{2}$-scheme. By Lemmas~\ref{lem:supportsrho} and~\ref{lem:hrho}, we know
that $_2\H^3_{\{0,1\}}(\AA^1_S,\Gm)^{S_3} = 0$. Thus from Lemma
\ref{lem:ptorsion}, we see that ${_2}\Br'(S) \to {_2}\Br'(X_S)^{S_3}$ is an
isomorphism. If we tensor the sequence \eqref{eq:Brseq} with $\Z_{(2)}$, we obtain
(using Lemmas \ref{lem:htriv} and \ref{lem:hrho}) the exact sequence
\begin{align}\label{eq:Br2}
    0 \to {_2\Br'(S)} \to {_2\Br'(\MM(2))^{S_3}} \to
    \H^1(S,\mu_2) \xrightarrow{\partial}\H^1(S_3,\, {_2\Br'(X_S)})
    \cong \Br'(S)[2] \to \cdots
\end{align}
Here we recall that we denote for an abelian group $A$ by $A[2]$ its $2$-torsion, while ${_2A}$ denotes its $2$-primary torsion. We want to analyze the boundary map $\partial$. 

%
%
 
\begin{lemma}\label{lem:boundary2}
    If $u\in\H^1(S,\mu_2)$, then $\partial(u)$ equals
    the Brauer class of the cyclic (quaternion) algebra $(-1,u)$.
\end{lemma}

\begin{proof}
   We assume that $S$ is connected. Denote by $\pi\colon X_S \to \B C_{2,X_S}$ the projection and by $c\colon \B C_{2,X_S} \to X_S$ the
  canonical map to the coarse moduli space.
  Denote by $$r\colon \Br'(\B C_{2,X_S}) \to \H^1(X_S, \mu_2)$$
  the map obtained from the Leray spectral sequence. Finally, let
  $$s\colon\H^1(X_S, \mu_2) \to \Br'(\B C_{2,X_S})$$
  be given by $s(u) = [(\chi,u)]=[(\chi,c^*u)]$, where $\chi \in \H^1(\B C_{2,X_S}, C_2)$ classifies $\pi$. We have $r(s(u)) = u$ by Lemma \ref{lem:cyclic2}. 
  
  Using \cite{serre-galois}*{Section 5.4}, we can compute a crossed homomorphism representing $\partial(u)$ thus as 
  $$g\mapsto \pi^*(g(s(u)) -s(u)) \in \Br'(X_S).$$ Consider the subgroup $C_2 = \langle
  (2\,3)\rangle \subset S_3$ and the $C_2$-equivariant morphism $z\colon S \to
  X_S$ classifying $y^2 = x(x-1)(x+1)$ (i.e.,\ $t=-1$). It follows from
  Lemmas~\ref{lem:htriv} and~\ref{lem:hrho} that the morphism
  $z^*\res_{C_2}^{S_3}$ induces an isomorphism $\H^1(S_3, {_2}\Br'(X_S)) \to\H^1(C_2,
  {_2}\Br'(S))$. The isomorphism $\H^1(C_2,
  \Br'(S)) \to \Br'(S)[2]$ is given by evaluating the crossed homomorphism at the
  non-trivial element $(2\,3)\in C_2$. Thus, the coboundary map $\partial:\H^1(S,\mu_2) \to \Br'(S)[2]$
  sends $u$ to 
  $$z^*\pi^*((2\,3)(s(u))-s(u)).$$
  As the pullback of $X_S \to \B C_{2,X_S}$ along $\pi\circ z\colon S \to X_S \to \B
  C_{2,X_S}$ is the trivial $C_2$-torsor, $z^*\pi^*s(u) = (\pi z)^*(\chi,u)$ defines the
  trivial Brauer class. By Lemma \ref{lem:S3action}, the action of $(2\,3)$
  multiplies the torsor $X_S \to \B C_{2,X_S}$ with the torsor $\B
  C_{2,X_S}(\sqrt{t}) \to \B C_{2,X_S}$. Thus, $z^*\pi^*(2\, 3)(s(u)) = (z^*t,u) =
  (-1,u)$.
\end{proof}
 
Summarizing, we obtain the following result.
 
\begin{proposition}\label{prop:BrInv}
    Let $S$ be a regular noetherian scheme over $\ZZf{2}$.
    We have an exact sequence
    \[0 \to {_{2}\Br'(S)} \to {_{2}\Br'(\MM(2)_S)}^{S_3} \to \H^1(S,\mu_2)
        \xrightarrow{\partial} \Br'(S)[2]\]
    with $\partial(u) = [(-1,u)]$. The map ${_{2}\Br'(S)} \to
    {_{2}\Br'(\MM(2)_S)}^{S_3}$ is
    non-canonically split. 
\end{proposition}

\begin{proof}
    The exact sequence is exactly \eqref{eq:Br2}. 
    The identification of $\partial(u)$ follows from the previous lemma. 
    For the splitting, choose an $S$-point $S \to \MM(2)_S$. Then the
    composition $\Br'(\MM(2)_S)^{S_3} \to \Br'(\MM(2)_S) \to \Br'(S)$ provides the
    splitting.
\end{proof}

\section{The $p$-primary torsion in $\Br(\Mscr_{\ZZ[\frac12]})$ for primes $p\geq 5$}\label{sec:p}

Before we proceed to study the $3$-primary and $2$-primary torsion, we will
show in this section that for a large class of $S$ there is no $p$-primary
torsion for $p\geq 5$ in the Brauer group of $\Mscr_S$. Lemma~\ref{lem:hrho} implies
the crucial fact that there are no $S_3$-invariant classes in
$_p\Br'(X_S)$ ramified at $\{0,1\}$ when $p\neq 3$ is invertible on $S$. The main point of the following
theorem is that this is true for $p\geq 5$ even for certain regular noetherian schemes
where $p$ is not a unit.

\begin{theorem}\label{thm:ptorsion}
    Let $S$ be a regular noetherian scheme over $\ZZ$ and $p\geq 5$ prime.
    Assume that $S[\frac1{2p}] = S_{\ZZf{2p}}$ is dense in $S$ and that $\Mscr_S \to S$ has a section.
    Then the natural map $_p\Br'(S)\rightarrow{_p\Br'(\Mscr_{S})}$ is an isomorphism.
\end{theorem}
\begin{proof}
Assume first that $2$ is invertible on $S$. 
    The only contribution to $_p\Br'(\Mscr_{S})$ in the Leray--Serre spectral
    sequence~\eqref{eq:lsss} occurs
    as
    \begin{equation}\label{eq:234}
        _p\left(\H^2(\B C_{2,X_S},\Gm)^{S_3}\right)=\left(_p\H^2(\B
        C_{2,X_S},\Gm)\right)^{S_3}
    \end{equation}
     because $\H^i$ of $S_3$ for $i\geq 1$ can never have $p$-primary torsion for $p\geq 5$. 
     
     We will argue that the $p$-group~\eqref{eq:234} is isomorphic to $_p\Br'(S)$ for all primes $p\geq 5$ if additionally $p$ is invertible on $S$. To do so, note first that 
     $$_p\H^2(\B
    C_{2,X_S},\Gm)\iso{_p\Br'(X_S)}$$ for $p\neq 2$ by Proposition~\ref{prop:bcn}.   
   By Lemmas \ref{lem:ptorsion}, \ref{lem:supportsrho} and \ref{lem:hrho}, we see that $_p\Br'(S) \to {_p\Br'(X_S)}^{S_3}$ is an isomorphism.
   
   This shows the theorem if $2p$ is invertible on $S$. Let now $S$ be arbitrary regular noetherian such that $S[\frac1{2p}] \subset S$ is dense and $\Mscr_S$ has an $S$-point. Consider the commutative diagram
   \[
   \xymatrix{{_p}\Br'(\Mscr_S) \ar[r]\ar[d] & {_p}\Br'(\Mscr_{S[\frac1{2p}]}) \ar[d]^{\cong} \\
   {_p}\Br'(S) \ar[r] & {_p}\Br'(S[\frac1{2p}]).
   }
   \] 
induced by the choice of an $S$-point of $\Mscr_S$. 
As $\Mscr_S$ has a cover by a scheme that is fppf over $S$ and fppf morphisms
are open \cite[Thm 2.4.6]{EGAIV.2}, $\Mscr_S \to S$ is open as well. Thus,
$\Mscr_{S[\frac1{2p}]} \subset \Mscr_S$ is dense and hence $\Br'(\Mscr_S) \to
\Br'(\Mscr_{S[\frac1{2p}]})$ is injective by Proposition \ref{prop:BrOmnibus}.
This implies that $\Br'(\Mscr_S) \to \Br'(S)$ is injective as well. As it is
also split surjective, we see that it is an isomorphism. 
\end{proof}

\begin{remark}
In general, it is a subtle question to decide whether $\Mscr_S$ has an $S$-point. For example for $S = \ZZf{2}$ or $S= \ZZf{3}$, there is such an $S$-point, but for $S = \ZZf{5}$ or $S = \ZZf{29}$ there is none (for this and other examples see \cite[Cor 1]{badreduction}). Nevertheless, sometimes one can still control the $p$-power torsion for $p\geq 5$ if there is no $S$-point, as the following corollary shows. 
\end{remark}

\begin{corollary}\label{cor:nop}
    The Brauer groups
    $\Br(\Mscr)\subseteq\Br(\Mscr_{\ZZf{2}})$ have only $2$ and $3$-primary torsion.
\end{corollary}

\begin{proof}
    Indeed, $\Br(\Mscr)\subseteq\Br(\Mscr_{\ZZf{2}})$, and there is
    no $p$-torsion in $\Br(\ZZf{2})\iso\ZZ/2$ for $p\neq 2$.
\end{proof}

\section{The $3$-primary torsion in $\Br(\Mscr_{\ZZ[\tfrac16]})$}\label{sec:3prim}

The next theorem describes the $3$-primary torsion in $\Br'(\Mscr_S)$
in many cases.

\begin{theorem}\label{thm:3torsion}
    Let $S$ be a regular noetherian scheme. If $6$ is a unit on $S$, then
    there is an exact sequence
    $$0\rightarrow{_3\Br'(S)}\rightarrow{_3\Br'(\Mscr_S)}\rightarrow\H^1(S,C_3)\rightarrow 0,$$ which is non-canonically split.
    The map ${_3\Br'(\Mscr_S)}\to\H^1(S,C_3)$ can be described as the
    composition of pullback to $X_S$ and taking the ramification at the divisor
    $\{0\}$ in $\mathds{A}^1_S$ defined by $t$ (using Proposition \ref{prop:puritycalc}). 
\end{theorem}

\begin{proof}
    The $3$-primary torsion
    in $\Br'(\Mscr(2)_S) \cong \Br'(\B C_{2,X_S})$ is just the $3$-primary torsion in
    $\Br'(X_S)$ by Proposition~\ref{prop:bcn} as $\H^1_{pl}(X_S; \mu_2)_{(3)} = 0$.
    Similarly, since $2$ is invertible in $S$, the Leray--Serre spectral
    sequence~\eqref{eq:lsss} together with the group cohomology
    computations of Lemma~\ref{lem:htriv} and Lemma~\ref{lem:hrho} and the exact sequences \eqref{eq:GmXS}-\eqref{eq:BrXS} say
    that 
    $$_3\Br'(\Mscr_{S})\iso\left(_3\Br'(X_{S})\right)^{S_3}.$$
    Since $3$ is invertible in $S$, we have a short exact sequence
    $$0\rightarrow{_3\Br'(S)}\rightarrow{_3\Br'(X_S)}\rightarrow{_3\H^3_{\{0,1\}}(\AA^1_S,\Gm)}\rightarrow
    0$$ 
    by Lemma~\ref{lem:ptorsion}.
    
    The $0$ and $1$ sections are disjoint, so that there is an isomorphism of
    $S_3$-modules
    $$_3\H^3_{\{0,1\}}(\AA^1_S,\Gm)\iso\bigoplus_{i=0,1}\H^1(S,\QQ_3/\ZZ_3)\cong \tilde{\rho}\otimes\H^1(S,\QQ_3/\ZZ_3)$$
    by Proposition~\ref{prop:puritycalc} and
    Lemma~\ref{lem:supportsrho}. Thus, Lemma~\ref{lem:hrho} implies that the long exact sequence
    in $S_3$-cohomology takes the following form:
    $$0\rightarrow{_3\Br'(S)}\rightarrow{_3\Br'(X_S)^{S_3}}\rightarrow\H^1(\Spec
    S,C_3)\rightarrow\H^1(S_3, {_3\Br'(S)}).$$ However, the action of $S_3$ on
    $_3\Br'(S)$ is trivial, so the group on the right vanishes by
    Lemma~\ref{lem:htriv}. Since $\Mscr_S$ has an $S$-point (because $2$ is inverted), the splitting follows.
    
    The map $_3\Br'(X_S) \to \bigoplus_{i=0,1}\H^1(S,\QQ_3/\ZZ_3)$ takes the ramification at the divisors $\{0\}$ and $\{1\}$.   
    At this point, we need to make the isomorphism
    $\left(\tilde{\rho}\otimes\H^1(S,\QQ_3/\ZZ_3)\right)^{S_3}\rightarrow\H^1(S,C_3)$
    more explicit. By choosing the ordered basis $t,t-1$ of
    $\Gm(X_S)/\Gm(S)\iso\tilde{\rho}$, we see from the description of the action
    that $\Z/3 \cong (\tilde{\rho}/3)^{S_3}\subseteq\tilde{\rho}/3$ is generated by
    $t(t-1)$. Indeed, $\sigma(t(t-1))=-t^{-2}(t-1)$ and $\tau(t(t-1))=-t^{-2}(t-1)$ for $\sigma = (1\,3\,2)$ and $\tau = (2\,3)$ and thus 
    $$\sigma(t(t-1)) \equiv t(t-1) \equiv \tau(t(t-1)) \in \Gm(X_S)/3.$$
    This implies that
    $\left(\tilde{\rho}\otimes\H^1(S,\QQ_3/\ZZ_3)\right)^{S_3}\rightarrow\H^1(S,C_3)$
    can be identified with projection onto the first coordinate. Thus,
    $_3\Br'(S_R)^{S_3} \to \H^1(S,C_3)$ takes the ramification at the divisor
    $\{0\}$. 
\end{proof}

We want to be more specific about the Azumaya algebras arising from
$\H^1(S,C_3)$. For that purpose consider the section $\Delta
\in\H^0(\Mscr,\lambda^{\tensor 12})$, which is defined as follows. Given an
elliptic curve $E$ over $S$, we can write it Zariski locally in
Weierstrass form. Consider its discriminant $\Delta_E \in \Oscr(S)$ and its invariant
differential $\omega \in \Omega^1_{E/R}(E) \cong \lambda(R)$. It is easy to see
by~\cite{silverman}*{Table~3.1}
that $\Delta = \Delta_E \omega^{\tensor 12}$ is a section of $\lambda^{\tensor
12}$, which is invariant under coordinate changes. Thus, $\Delta$ defines a
section of $\lambda^{\tensor 12}$ on $\Mscr$.

\begin{lemma}\label{lem:3section}
    By Construction \ref{const:root}, we can associate with the line bundle $\Lscr
    = \lambda^{\tensor 4}$ and the trivialization $\Delta\colon \Oscr_{\MM} \to
    \Lscr^{\tensor 3}$ the $\mu_3$-torsor $\Mscr(\sqrt[3]{\Delta}) \to \Mscr$
    whose class in $\H^1(\Mscr, \mu_3)$ we denote by $[\Delta]_3$.
    If $S$ is a regular noetherian scheme and $6$ is a unit on $S$, then the composite
    $$\H^1(S, C_3) \to\H^1(\Mscr_S, C_3) \xrightarrow{\cup (-[\Delta]_3)}\H^2(\Mscr_S, \mu_3) \to {_3\Br'(\Mscr_S)}$$
    is a section of the map ${_3\Br'(\Mscr_S)} \to\H^1(S, C_3)$ of
    Theorem~\ref{thm:3torsion}. 
\end{lemma}

\begin{remark}
    Informally, this section associates with $\chi \in\H^1(S, C_3)$ the symbol algebra $[(\chi, \Delta^{-1})_3]$. 
\end{remark}

\begin{proof}
    The pullback of $\Delta$ to $X$ is the discriminant of the universal
    Legendre curve, which is $16t^2(t-1)^2$ (using the standard trivialization of
    $\lambda$ on $X$ given by $\frac{dx}{2y}$). For $\chi \in\H^1(S, C_3)$
    the pullback of $[(\chi, \Delta^{-1})_3]$ to $\Br'(X_S)$ is thus $[(\chi,
    4t(t-1))_3]$. As $4t(t-1)$ is a uniformizer for the local ring of $\AA^1_S$
    at $t=0$, Proposition~\ref{prop:ramification} implies the result.
\end{proof}

\begin{remark}
    While this map $\H^1(S,C_3)\rightarrow\Br'(\Mscr_S)$ is defined
    whether or not $6$ is a unit on $S$, without this assumption we do not know
    that $\H^1(S,C_3)$ is the cokernel of
    $_3\Br'(S)\rightarrow{_3\Br'(\Mscr_S)}$.
\end{remark}

\begin{corollary}\label{cor:3tors}
    When $S=\Spec\ZZ[\tfrac16]$, there is an isomorphism
    $_3\Br(\Mscr_{\ZZf{6}})\iso\QQ_3/\ZZ_3\oplus\ZZ/3$. 
    The $3$-torsion subgroup is generated by classes
    $\sigma$ and $\theta$, which can be described as follows.
    Let $\chi \in\H^1(\Spec \ZZf{6}, C_3)$ be the character of the Galois extension  $\QQ(\zeta_9 + \overline{\zeta}_9)$ of $\Q$.
    Then $\sigma = [(\chi, 6)_3]$ and $\theta = [(\chi, 16\Delta^{-1})_3]$,
    which pulls back to $[(\chi, t(t-1))_3]$ on $X_{\ZZf{6}}$. 
\end{corollary}

\begin{proof}
    We claim first that $\QQ(\zeta_9 + \overline{\zeta}_9)$ is the only cyclic
    cubic extension $L$ of $\QQ$ that ramifies at most at $2$ and $3$. This can
    either be deduced from \cite[I.\S 1.2]{hasse} or shown as follows. By the
    Kronecker--Weber theorem, any cyclic cubic extension $L$ of $\QQ$ has to embed into a cyclotomic extension
    $\QQ(\zeta_n)$ and is more precisely its fixed field under a normal subgroup
    $H\subset (\Z/n)^\times$ of index $3$. As $H$ contains all elements of
    $2$-power order, we can assume that $n$ is odd and thus $L$ does not ramify
    at $2$ but only at $3$. Proposition 3.1 of \cite{lemmermeyer} shows that
    $L$ is unique and must be $\QQ(\zeta_9 + \overline{\zeta}_9)$. This implies
    that $\H^1(\Spec \ZZf{6}, C_3) \cong \ZZ/3$.
    
    Using that $_3\Br(\ZZ[\tfrac16])\iso\QQ_3/\ZZ_3$, the structure of the
    Brauer group $\Br(\Mscr_{\ZZf{6}})[3]$ follows from Theorem \ref{thm:3torsion}. The description of $\theta$ follows directly from the last lemma (where we have modified the section by an element of $_3\Br(\ZZf{6})$ for convenience). 
    
    Last we need to show that $[(\chi,6)_3]$ is non-zero in $\Br(\ZZf{6})[3]\cong \ZZ/3$. It suffices to check that $(\chi,6)$ is ramified at the prime $(2)$. Note that the minimal polynomial of $\zeta_9 + \overline{\zeta}_9$ is $w^3+w+1$. 
    By Proposition~\ref{prop:ramification}, the ramification at $(2)$ in
    $\H^1(\FF_2,C_3)\iso\ZZ/3$
    is the class of the extension $w^3+w+1$ over $\FF_2$. Since this polynomial
    is irreducible (it has no solutions in $\FF_2$ and it has degree $3$), it
    follows that the ramification is non-zero.
\end{proof}

\begin{corollary}
    Let $R=\ZZ[\tfrac12]$ or $R=\ZZ$.
    Then the order of $_3\Br(\Mscr_R)$ is either $1$ or $3$. 
\end{corollary}

\begin{proof}
    Using the injectivity of $\Br(\Mscr)\rightarrow\Br(\Mscr_{\ZZ[\tfrac12]})$,
    it suffices to prove this when $R=\ZZ[\tfrac12]$. Now, we claim that no
    non-zero class
    $\alpha\in{_3\Br(\ZZ[\tfrac16])}\subseteq{_3\Br(\Mscr_{\ZZ[\tfrac16]})}$
    extends to $\Mscr_{\ZZ[\tfrac12]}$. Indeed, we can take the Legendre curve
    $y^2=x(x-1)(x-2)$, which defines a point
    $\Spec\ZZ[\tfrac12]\rightarrow\Mscr$. Using the commutative diagram
    $$\xymatrix{
        \Spec\ZZ[\tfrac16]\ar[r]\ar[d]  &   \Mscr_{\ZZ[\tfrac16]}\ar[r]\ar[d] &   \Spec\ZZ[\tfrac16]\\
        \Spec\ZZ[\tfrac12]\ar[r]        &   \Mscr_{\ZZ[\tfrac12]}, &
    }$$
    we see that if $\alpha\in{_3\Br(\ZZf{6})}$ did extend to $\Mscr_{\ZZ[\tfrac12]}$, then
    it would be zero in the Brauer group of $\Spec\ZZ[\tfrac16]$, as it
    would extend to $_3\Br(\ZZ[\tfrac12])=0$. However, these classes are
    all non-zero in $\Br(\ZZ[\tfrac16])$ since the composition at the top of
    the commutative diagram is the identity for any $\ZZ[\tfrac16]$-point of
    the moduli stack.

    Now, if $\alpha=\beta+m\theta$ extends, where
    $\beta\in{_3\Br(\ZZ[\tfrac16])}$,
    then $3\alpha=3\beta$ also extends. Hence, it must be that $\alpha$ has
    order at most $3$. In particular, this means that every class of
    $_3\Br(\Mscr_{\ZZ[\tfrac12]})$ is actually $3$-torsion and hence this group
    is a subgroup of $(\ZZ/3)^2$. But, we have already seen that $\sigma$ does
    not extend. So, it is a proper subgroup, and hence it has order at most
    $3$.
\end{proof}

\begin{proposition}\label{prop:3sufficiency}
    Suppose there are Legendre curves $E_i:y^2=x(x-1)(x-t_i)$ over
    $\Spec\ZZ_3[\zeta_3]$ for $i=1,2$ such that
    \[ [(\chi,t_1(t_1-1))]\neq 0 \text{ and } [(\chi,t_2(t_2-1))]=0\]
     in $\Br(\Q_3(\zeta_3))[3] = \Z/3$, where $\chi$ denotes the pullback of the Galois extension $\QQ(\zeta_9 + \overline{\zeta}_9)$ of $\QQ$ to $\QQ_3(\zeta_3)$. Then $_3\Br(\Mscr[\tfrac12])=0$.
\end{proposition}

\begin{proof}
    Suppose that $\alpha=a\sigma+b\theta$ is a linear combination of the
    classes found in Corollary~\ref{cor:3tors} where we can assume that
    $b\in\{1,2\}$ since $\sigma$ does not extend. Suppose
    that $\alpha$ extends to
    $\Br(\Mscr_{\ZZ[\tfrac12]})$. We can pull back $\alpha$ along the two
    $\QQ_3(\zeta_3)$-points of $\Mscr$ defined by $E_i$ and compute the ramifications in $\Br(\QQ_3(\zeta_3))$. Let $k=[(\chi,t_1(t_1-1))]$. For $i=1$, we get
    $a+bk$ and for $i=2$ we get $a$. Since $k$ is non-zero, these cannot be
    simultaneously zero modulo $3$.
    But the two maps $\Br(\Mscr_{\ZZ[\tfrac12]}) \to \Br(\QQ_3(\zeta_3))$
    factor over $\Br(\ZZ_{3}[\zeta_3]) = 0$, which is a contradiction. Thus,
    $\alpha$ cannot extend to $\Br(\Mscr_{\ZZ[\tfrac12]})$.
\end{proof}

\section{The ramification of the $3$-torsion}
Our aim in this section is to show that $_3\Br(\Mscr[\tfrac12])=0$ using Proposition \ref{prop:3sufficiency}. 

\begin{lemma}\label{lem:Kummer}
    The natural inclusion
    $\QQ(\zeta_9+\overline{\zeta}_9,\zeta_3)\rightarrow\QQ(\zeta_9)$ is an
    isomorphism.
\end{lemma}
\begin{proof}
The left hand side is a subfield of $\Q(\zeta_9)$ that is strictly larger than
$\Q(\zeta_9 + \overline{\zeta}_9)$ and thus is equal to $\Q(\zeta_9)$.
\end{proof}
 
We will need the following lemma to aid our Hilbert symbol calculations below.

\begin{lemma}\label{lem:trace}
 Consider the cyclotomic field $\Q(\zeta)$ with $\zeta = \zeta_3$ and $\pi =
 1-\zeta$ and denote by $\Tr$ the trace for $\QQ(\zeta)$ over $\QQ$. Then 
 $$ \Tr(\pi^{6k+l}) = \begin{cases}
                       (-1)^{3k}\cdot 3^{3k}\cdot 2 & \text{ if } l=0 \\
                       (-1)^{3k}\cdot 3^{3k+1}  & \text{ if } l=1 \\
                       (-1)^{3k}\cdot 3^{3k+1} & \text{ if } l=2 \\
                       0 & \text{ if } l=3 \\
                        (-1)^{3k+1}\cdot 3^{3k+2} & \text{ if } l=4 \\
                         (-1)^{3k+1}\cdot 3^{3k+3} & \text{ if } l=5 \\                       
                      \end{cases} $$
\end{lemma}

\begin{proof}
 We have $\pi^2 = (1-\zeta)^2 = -3\zeta$ and thus $\pi^{6k} = (-3)^{3k}$. Therefore, we have just to compute $\Tr(\pi^l)$ for $l=0,\dots, 5$, which is easily done. 
\end{proof}

We come to a key arithmetic point in our proof, where we compute the Hilbert
symbol at the prime $3$ of certain degree $3$ cyclic algebras. By Proposition
\ref{prop:Hilbert}, this will allow us to check whether certain cyclic algebras
are zero in the Brauer group.

\begin{lemma}\label{lem:Hilbert}
    Consider the cyclotomic field $\Q_3(\zeta)$ with $\zeta = \zeta_3$ and $\pi = 1-\zeta$ the uniformizer. Then we have
    $$\binom{\zeta, t(t-1)}{\pi} = \zeta^{1-b^2}$$
    in $\mu_3(\Q_3(\zeta))$, where $t = 2 + b\pi$ with $b\in \ZZ_3$.
\end{lemma}

\begin{proof}
    We use the formula of Artin--Hasse (see~\cite{neukirch}*{Theorem V.3.8}) to
    compute this Hilbert symbol. By this formula, we have
    $$\binom{\zeta, a}{\pi} = \zeta^{\Tr(\log a)/3},$$
    where $a \in 1 + \mathfrak{p}$ (for $\mathfrak{p} \subset \ZZ_3[\zeta_3]$ the
    maximal ideal) and $\Tr$ the trace for $\Q_3(\zeta)$ over $\QQ_3$. 

    This formula directly applies to $t-1 = 1 + b\pi$. We have
    $$ \log(t-1) = \sum_{i=1}^\infty (-1)^{i+1}\frac{(b\pi)^i}i.$$
    Again, it follows easily from Lemma \ref{lem:trace} that $\frac{\Tr(\pi^i)}{i}$ is divisible by $9$ for $i\geq 3$. Thus, 
    $$ \frac{\Tr(\log(t-1))}3 \equiv \frac{\Tr(b\pi)}3 - \frac{\Tr(b^2\pi^2)}{6} \equiv b - \frac{b^2}2 \mod 3$$
    and $\binom{\zeta, t-1}{\pi} = \zeta^{b-\frac{b^2}2}$.

    To compute $\binom{\zeta, t}{\pi}$ note that $\binom{\zeta, t}{\pi} =
    \binom{\zeta, -t}{\pi}\binom{\zeta, -1}{\pi} = \binom{\zeta, -t}{\pi}$.
    Indeed, $\binom{\zeta, -1}{\pi}^2 = 1$ and hence also $\binom{\zeta,
    -1}{\pi} = 1$ (in
    $\mu_3(\QQ_3(\zeta_3))$). We have
    $-t = 1 + (-3-b\pi)$. Thus, 
    $$ \log(-t) = \sum_{i=1}^\infty (-1)^{i+1}\frac{(-3-b\pi)^i}i.$$
    It follows easily from Lemma \ref{lem:trace} that $\frac{\Tr((-3-b\pi)^i))}{i}$ is divisible by $9$ for $i\geq 3$. Thus, 
    $$ \frac{\Tr(\log(-t))}3 \equiv \frac{\Tr(-3-b\pi)}3 - \frac{\Tr((-3-b\pi)^2)}{6} \equiv -2-b - \frac{b^2}2 \mod 3.$$
    Thus, $\binom{\zeta, t}{\pi} = \binom{\zeta, -t}{\pi} = \zeta^{-2-b-\frac{b^2}2}$.
    It follows that 
    $$\binom{\zeta, t(t-1)}{\pi} = \binom{\zeta, t}{\pi}\binom{\zeta, t-1}{\pi} = \zeta^{-2-b^2} =
    \zeta^{1-b^2},$$ as desired.
\end{proof}

\begin{theorem}\label{thm:Br3}
    We have 
    $${_3\Br(\Mscr)}=_3\Br(\Mscr_{\Z[\frac12]})  = 0.$$
\end{theorem}

\begin{proof}
    By Proposition~\ref{prop:3sufficiency} and Proposition~\ref{prop:Hilbert}, it suffices to find two Legendre
    curves $E_1$ and $E_2$ over $\ZZ_3[\zeta_3]$ with corresponding classes
    $[(\chi,t_1(t_1-1))]\neq 0$ and $[(\chi,t_2(t_2-1))]=0$. The associated
    condition on the Hilbert
    symbols is $\binom{\zeta,t_1(t_1-1)}{\pi}\neq 1$ and
    $\binom{\zeta,t_2(t_2-1)}{\pi}=1$. (Recall here that $\chi$ is the
    character associated with adjoining $\zeta_9+\overline{\zeta}_9$ which over
    $\QQ_3(\zeta_3)$ is isomorphic to $\QQ_3(\zeta_9)$ by Lemma~\ref{lem:Kummer}.) Take $t_i=2+b_i\pi$,
    where $b_1=0$ and $b_2=1$.
    Consider the two elliptic curves $$E_1:y^2=x(x-1)(x-2)$$ and
    $$E_2:y^2=x(x-1)(x-(2+\pi)).$$ The previous lemma says that we have
    $$\binom{\zeta,t_1(t_1-1)}{\pi}=\binom{\zeta,2(2-1)}{\pi}=\zeta\neq 1$$ and
    $$\binom{\zeta,t_2(t_2-1)}{\pi}=\binom{\zeta,(2+\pi)(1+\pi)}{\pi}=\zeta^0=1.$$
    This completes the proof.
\end{proof}

\section{The $2$-primary torsion in $\Br(\Mscr_{\ZZ[\frac{1}{2}]})$}\label{sec:2torsion}

Throughout this section, let $S$ denote a connected regular noetherian scheme over
$\Spec\ZZ[\frac{1}{2}]$. Given a stack $X$ over $S$, let
$\overline{\Br}'(X)=\mathrm{coker}(\Br'(S)\rightarrow\Br'(X))$.

\begin{theorem}\label{thm:brm}
    Let $S$ be a regular noetherian scheme over $\ZZ[\tfrac12]$ with
    $\Pic(S)=0$. There is
    a natural exact sequence
    $$0 \to \Gm(S)/2 \to {_2}\overline{\Br}'(\Mscr_S) \to G \to 0,$$
    where $G \subset \Gm(S)/2$ is the subgroup of all those
    $u$ with $[(-1,u)] = 0\in \Br(S)$.
\end{theorem}

We will prove the theorem after several preliminaries.
Figure~\ref{fig:m11} shows a small part of the $2$-local Leray--Serre spectral
sequence~\eqref{eq:lsss} for the $S_3$-Galois cover $\Mscr(2)_S\rightarrow\Mscr_S$.
The description follows from Lemmas \ref{lem:htriv}, \ref{lem:hrho},
\ref{lem:twolines} and Proposition \ref{prop:BrInv}.

\begin{figure}[h]\label{fig:DSS}
    \centering
    \begin{equation*}
        \xymatrix@R=3pt{
            _2\Br'(S)\oplus \Pic(S)[2]\oplus G\ar[rrd]&&&\\
            \Pic(S)_{(2)}\oplus\mu_2(S)\ar[rrd]&\Pic(S)[2]\oplus\mu_2(S)\ar[rrd]^{}&\Pic(S)/2\oplus\mu_2(S)&\\
            \Gm(S)_{(2)}&\mu_2(S)&\Gm(S)/2&\mu_2(S)\\
        }
    \end{equation*}
    \caption{The $\E_2$-page of the Leray--Serre spectral sequence computing
    $\H^i(\Mscr_S,\Gm)_{(2)}$ for $i\leq 2$.}
    \label{fig:m11}
\end{figure}

From now on, we will localize everything in this section implicitly at $2$. 

\begin{proposition}
    Let $S$ be a connected regular noetherian scheme over $\ZZf{2}$.
    The differential $d_2^{0,1}$ in the Leray--Serre
    spectral sequence of Figure~\ref{fig:m11} always vanishes and $d_2^{0,2}$ and $d_3^{0,2}$ vanish if $\Pic(S)=0$.
\end{proposition}

\begin{proof}
    The map 
    $$\Pic(\Mscr_S) \to \E_2^{0,1}\iso{\Pic(S)_{(2)}\oplus\mu_2(S)}$$
    is surjective as $-1 \in \mu_2(S)$ can be realized as $\lambda^{\tensor 6}$. This implies that there can be no differential originating from
    $\E_2^{0,1}$. Moreover, $_2\Br'(S)$ splits off
    from $\Br'(\Mscr_S)$, so the differentials $d_2^{0,2}$ and $d_3^{0,3}$
    vanish on $_2\Br'(S)$. 
    
    Now assume $\Pic(S) = 0$. Then also $\Pic(S)[2]=0$ and $\Pic(S)/2=0$. So, we are concerned
    with the vanishing of
    $d_2^{0,2}:G\rightarrow\mu_2(S)$ and $d_3^{0,2}: G\rightarrow \mu_2(S)$. However, by pulling back the
    spectral sequence to a geometric point $\overline{x}$ of $S$, we
    find $G = \Gm(\overline{x})/2=0$, while $\mu_2(\overline{x})\iso\ZZ/2$. This
    implies that $d_2^{0,2}$ and $d_3^{0,2}$ vanish.
\end{proof}

To resolve the differential $d_2^{1,1}$ and solve possible extension issues we
will first consider schemes $S$ over $\Z[\frac12,i]$. In this case we can
compare the Leray--Serre spectral sequence considered above with the
Leray--Serre
spectral sequence for the $C_2$-Galois cover $\B C_{2,S} \to \B C_{4,S}$.

\begin{proposition}\label{prop:mi}
    If $S$ is a regular noetherian $\ZZ[\frac{1}{2},i]$-scheme, then $\Br'(\Mscr_S)\iso\Br'(\B C_{4,S})$.
\end{proposition}

\begin{proof}
Consider the elliptic curve $E: y^2 = x(x-1)(x+1)$
over $\ZZ[\frac12,i]$ with discriminant $64$.
It has an automorphism $\eta$ of order $4$ given by $y \mapsto iy$ and $x\mapsto -x$,
which defines a map $\B C_{4,\ZZ[\frac12,i]} \to \Mscr_{\ZZf{2}}$. The $2$-torsion points of $E$ are $(0,0)$, $(1,0)$ and $(-1,0)$; taking them in this order defines a full level $2$ structure. We can base change this elliptic curve together with its level structure to an arbitrary $\ZZ[\frac12,i]$-scheme $S$. This
results in pullback squares
\begin{equation}\label{eq:pullback}
    \xymatrix{
        \coprod_{S_3}S\ar[r]\ar[d] &\coprod_{S_3/C_2}\B C_{2,S}\ar[d] \ar[r] & \Mscr(2)_S \ar[d] \\
        S\ar[r]&\B C_{4,S} \ar[r] & \Mscr_{S}.}
\end{equation}
Here we use that $\eta$ acts on the scheme $\coprod_{S_3}S$ of level
structures on $E_S$ by multiplication with the cycle $(2\,3) \in S_3$ and in
particular $\eta^2$ acts trivially. Thus, the stack quotient
$\left(\coprod_{S_3}S\right)/C_4$ is equivalent to $\coprod_{S_3/C_2}\B C_{2,S}$. More
precisely, the $S_3$-Galois cover $\coprod_{S_3/C_2}\B C_{2,S} \to \B C_{4,
S}$ is induced along an inclusion $C_2 \to S_3$ from the $C_2$-Galois cover $\B
C_{2, S} \to \B C_{4,S}$. The right square is indeed cartesian as can be
checked after base change along the \'etale cover $S\to \B C_{4,S}$.

In the Leray--Serre spectral sequence
$$\E_2^{p,q}=\H^p(S_3,\H^q(\coprod_{S_3/C_2}\B
C_{2,S},\Gm))\Rightarrow\H^{p+q}(\B C_{4,S},\Gm),$$ 
the $S_3$-modules $\H^q(\coprod_{S_3/C_2}\B C_{2,S})$ are all induced up from $C_2$. Thus, the
spectral sequence is isomorphic to the Leray--Serre spectral sequence for the $C_2$-Galois cover $\B C_{2, S} \to \B C_{4,S}$. 

The Leray--Serre spectral sequence computing $\H^{p+q}(\B C_{4,S},\Gm)$ from
the $\Gm$-cohomology of $\B C_{2,S}$ is displayed in
Figure~\ref{fig:c4c2}. The computation follows from Proposition~\ref{prop:bcn}
together with the fact that $C_2$ acts trivially on the cohomology of $\B
C_{2,S}$ (as indeed the morphism $t\colon \B C_{2,S} \to \B C_{2,S}$ for $t\in C_2$ the generator is the identity; only the natural transformation $\id \to t^2$ is not the identity).

\begin{figure}[h]
    \centering
    \begin{equation*}
        \xymatrix@R=3pt{
            \Br'(S)\oplus\Pic(S)[2]\oplus\Gm(S)/2\ar[rrd]^{}&&&\\
            \Pic(S)\oplus\mu_2(S)&\Pic(S)[2]\oplus\mu_2(S)\ar[rrd]^{}&\Pic(S)/2\oplus\mu_2(S)&\\
            \Gm(S)&\mu_2(S)&\Gm(S)/2&\mu_2(S)\\
        }
    \end{equation*}
    \caption{Part of the Leray--Serre spectral sequence for $\B C_{2,S}\rightarrow\B
    C_{4,S}$.}
    \label{fig:c4c2}
\end{figure}

By the considerations above, the pullback square~\eqref{eq:pullback} induces a map 
$$\H^p(S_3,\H^q(\MM(2)_S,\Gm)) \to\H^p(C_2,\H^q(\B C_{2,S},\Gm)).$$
Note first that $G = \Gm(S)/2$ in our case as $-1$ is a square. If we identify
$\MM(2)_S$ with $\B C_{2,X_S}$ this map on cohomology groups is induced by the
maps $S \to X_S$ (classifying the Legendre curve $E_S$) and $C_2 \to S_3$. This
induces an isomorphism of spectral sequences for $p+q \leq 3$ and $q \leq 1$
for $p+q =3$ by Figure~\ref{fig:m11}.
\end{proof}

\begin{corollary}
    Let $S$ be a regular noetherian scheme over $\ZZ[\tfrac12]$.
    The restriction of the differential $d_2^{1,1}$ to $\mu_2(S)$ in the Leray--Serre spectral sequence for
    $\Mscr(2)_S\rightarrow\Mscr_S$ defines an isomorphism $\mu_2(S) \xrightarrow{\cong} \mu_2(S)$, while
    $d_3^{0,2}=0$.
\end{corollary}

\begin{proof}
    Consider $S'=\Spec\ZZ[\tfrac12,i]$. By Proposition~\ref{prop:bcn}, we see
    that 
    $$\Br'(\B C_{4,S'})\iso\Br'(S')\oplus \Pic(S')[4]\oplus\Gm(S')/4\iso\Gm(S')/4,$$
    since the Brauer and Picard groups of $\ZZ[\tfrac12,i]$ are zero. Hence, $\Br'(\B
    C_{4,S'})\iso\ZZ/4\oplus\ZZ/4$, with generators given as $i$ and $1+i$. In
    Figure~\ref{fig:c4c2}, we see that the only way to have a group of order 16
    in the abutment $\Br'(\B C_{4,S'})$ is that
    $d_2^{1,1}:\mu_2(S')\rightarrow\mu_2(S')$ is an isomorphism, both in the
    Leray--Serre spectral sequence for $\H^*(\B C_{4,S'}, \Gm)$ and for
    $\H^*(\MM_{S'}, \Gm)$. It follows that
    this differential is already an isomorphism in the Leray--Serre spectral
    sequence for $\H^*(\MM_{\ZZ[\tfrac12]}, \Gm)$ by
    naturality. This in turn implies by naturality that $d_2^{1,1}|_{\mu_2(S)}\colon \mu_2(S) \to \mu_2(S)$ in the
    Leray--Serre spectral sequence for $\H^*(\MM_{S}, \Gm)$ is an isomorphism for any regular noetherian
    $\ZZ[\tfrac12]$-scheme $S$. As the target of $d_3^{0,2}$ is already
    zero on $\E_3$, the differential $d_3^{0,2}$ must vanish.
\end{proof}

Finally, we prove the theorem from the beginning of the section.

\begin{proof}[Proof of Theorem~\ref{thm:brm}]
    The claim follows from the determination of the differentials in the
    range pictured in Figure~\ref{fig:m11}.
\end{proof}

We want to be more specific about the Brauer group classes coming from $\Gm(S)/2$. Recall the section $\Delta \in \H^0(\Mscr, \lambda^{\tensor 12})$ from Section \ref{sec:3prim}. 
As in Construction \ref{const:root}, we can define the $C_2$-torsor
$\Mscr(\sqrt{\Delta}) =
\mathbf{Spec}_{\Mscr_{\ZZf{2}}}\left(\bigoplus_{i\in\Z}\lambda^{\tensor
6i}/(\Delta -1)\right) \to \Mscr_{\ZZf{2}}$ that adjoins a square root of $\Delta$ to
$\Mscr_{\ZZf{2}}$. For a unit $u\in \Gm(\ZZ[\frac12])$, we denote by $(\Delta,
u)$ the symbol (quaternion) algebra associated with this torsor.

\begin{proposition}\label{prop:2cyclic}
 Let $S$ denote a connected regular noetherian scheme over
$\Spec \ZZf{2}$. Then the map 
$$\Gm(S)/2 \to \Br'(\Mscr_S)$$
from the Leray--Serre spectral sequence sends $u$ to $[(u, \Delta)_2]$. 
\end{proposition}
\begin{proof}
 We consider the Leray--Serre spectral sequence
 \[
  \H^p(C_3, \H^q(\Mscr(2)_S, \Gm))_{(2)} \Rightarrow \H^{p+q}(\Mscr(2)_S/C_3, \Gm)_{(2)},
 \]
where $\Mscr(2)/C_3$ denotes the stack quotient by the subgroup $C_3 \subset
S_3$. Its $\E_2$-term is clearly concentrated in the column $p=0$. From Lemma 
\ref{lem:twolines} and Proposition \ref{prop:BrInv}, it is easy to see that
$\H^0(C_3, \H^q(\Mscr(2),\Gm)_{(2)} \cong \H^0(S_3, \H^q(\Mscr(2),\Gm)_{(2)})$.
We can now consider the further Leray--Serre spectral sequence
 \[
  \H^p(C_2, \H^q(\Mscr(2)_S/C_3, \Gm))_{(2)} \Rightarrow \H^{p+q}(\Mscr_S, \Gm)_{(2)},
 \]
 and we see that it has the same $E_2$-term as the Leray--Serre spectral sequence for the $S_3$-cover $\Mscr(2)_S \to \Mscr_S$ in the range depicted in Figure \ref{fig:DSS}.
 
 We claim that the $C_2$-torsor $\Mscr(2)_S/C_3 \to \Mscr_S$ agrees with $\Mscr_S(\sqrt{\Delta}) \to \Mscr_S$. For this it suffices to show that $\Delta$ becomes a square on $\Mscr(2)_S/C_3$. With $p,q \in \H^0(\Mscr(2),\lambda^{\tensor 2})$ as in the discussion after Corollary \ref{cor:coarse}, we have $\Delta = 16p^2q^2(p-q)^2$. The $C_3$-action permutes $p,(-q)$ and $(q-p)$ cyclically so that $4pq(p-q)$ is a $C_3$-invariant section of $\lambda^{\tensor 6}$ whose square is indeed $\Delta$.  
 
 
 Now the statement follows from Corollary \ref{cor:cyclicgeneral}. 
\end{proof}

The following is one of our main results. We recall the convention that everything is implicitly $2$-local so that $\Br(\Mscr_{R})$ for a ring $R$ denotes really $\Br(\Mscr_R)_{(2)}$.

\begin{proposition}\label{prop:brQ}
 Let $P$ be a set of prime numbers including $2$ and denote by $\ZZ_P \subset \Q$ the
    subset of all fractions where the denominator is only divisible by primes
    in $P$. Then
    $$\Br(\Mscr_{\ZZ_P})\;\cong\; \Br(\ZZ_P)\quad \oplus
    \bigoplus_{\substack{p\in P\cup\{-1\},\\ p\equiv 3 \bmod 4}}\Z/2\quad\oplus
    \bigoplus_{\substack{p\in P,\\ p\not\equiv 3 \bmod 4}} \Z/4.$$
\end{proposition}
\begin{proof}
 First we have to compute the subgroup 
    $$G\subset \Gm(\ZZ_P)/2 \cong   \bigoplus_{P\cup \{-1\}} \FF_2.$$
    By Proposition~\ref{prop:Hilbert}, a quaternion algebra $(a,b)$
    ramifies at $p$ if and only if the Hilbert symbol $\binom{a,b}{p}$ equals $-1$. By
    \cite[Theorem V.3.6]{neukirch}, $\binom{-1,-1}{p} = -1$ if and only if $p=2, \infty$ and
    $\binom{-1,q}{p} = -1$ if and only if $q\equiv 3\mod 4$ and $p=2,q$ (for $q$ a prime
    number). We see that $G$ has an $\FF_2$-basis given by the primes not
    congruent to $3$ mod $4$. 
    
    We obtain a diagram
      \[\xymatrix{
     0 \ar[r] & \Gm(\ZZ_P)/2 \ar[d]\ar[r] & \overline{\Br}(\Mscr_{\ZZ_P}) \ar[r]\ar[d] & G \ar[r]\ar[d] & 0 \\
      0 \ar[r] & \Gm(\ZZ_P[i])/2 \ar[r] & \overline{\Br}(\Mscr_{\ZZ_P[i]})
          \ar[r] & \Gm(\ZZ_P[i])/2 \ar[r] & 0.
     }
     \] 
    By Proposition \ref{prop:mi}, $\overline{\Br}(\Mscr_{\ZZ[\tfrac12,i]}) \cong \Gm(\ZZ_P[i])/4$.  
    As the map $G\to \Gm(\ZZ_P[i])/2$ is injective, we see that none of the
    nonzero lifts of elements of $G$ to $\overline{\Br}(\Mscr_{\ZZ_P})$ are
    $2$-torsion. The proposition follows.
\end{proof}

This shows the $2$-local part of the computation of $\Br(\Mscr_{\QQ})$ and $\Br(\Mscr_{\ZZ[\tfrac12]})$ in Theorem \ref{thm:intromain}, while the $3$-local part was already contained in Theorems \ref{thm:3torsion} and \ref{thm:Br3} and the $p$-local part for $p>3$ in Theorem \ref{thm:ptorsion}.

\begin{remark}\label{rem:4cyclic}
    We can describe all the Brauer classes in $\overline{\Br}(\Mscr_{\ZZ_P})$
    explicitly when $P$ is again a set of prime numbers including $2$. We
    already saw in the last two propositions that $\overline{\Br}(\Mscr_{\ZZ_P})[2]$
    has an $\FF_2$-basis given by $[(p,\Delta)_2]$, where $p \in P \cup
    \{-1\}$. When $p\in S$ and either $p=2$ or $p\equiv 1\mod 4$,
    we will give explicit elements of order $4$ in the Brauer group of $\Mscr_{\ZZ_P}$ 
    generating the $\ZZ/4$-subgroups of the proposition.
 
 To describe the $4$-torsion we start with a small observation. Given a cyclic
 algebra $(\chi, \upsilon)$, where $\chi$ is a $C_4$-torsor and $\upsilon$ a
 $\mu_4$-torsor, $2[(\chi, \upsilon)]$ is represented by $(\chi', \upsilon')$,
 where $\chi'$ and $\upsilon'$ are obtained from $\chi$ and $\upsilon$ via the
 morphisms $C_4\to C_2$ and $\mu_4\to \mu_2$. Concretely, this means that
 $\chi'$ are the $C_2$-fixed points of $\chi$ and that if $\upsilon$ is given
 by adjoining the $4$-th root of a section $u$ of $\Lscr^{\otimes 4}$, then $\upsilon'$ is
 given by adjoining a square root of $u$, a section of $(\Lscr^{\otimes
 2})^{\otimes 2}$. 
 
 For primes $p \equiv 1 \mod 4$, we construct a $C_4$-Galois extension $L$ of
 $\QQ$ whose $C_2$-fixed points are $\QQ(\sqrt{p})$. As $\sqrt{p}$ is the
 Gau{\ss} sum $\sum_{a =1}^{p-1} \legendre{a}{p} \zeta_p^a$, we see that
 $\QQ(\sqrt{p}) \subset \QQ(\zeta_p)$. The Galois group of the $\QQ$-extension
 $\QQ(\zeta_p)$ is cyclic of order $p-1$, which is divisible by $4$. Thus, it
 has a unique cyclic subextension $L$ of degree $4$ whose $C_2$-fixed points
 are $\QQ(\sqrt{p})$. Note that $L$ is only ramified at $p$. Explicitly, $L$ is
 generated by the Gau{\ss} sum $\sum_{a = 1}^{p-1} \varphi(a) \zeta_p^a$, where
 $\varphi\colon (\Z/p)^{\times} \to \mu_4(\CC)$ is a surjective character. 
 
 For $p=2$, we take $L = \QQ(\zeta_{16} + \overline{\zeta}_{16})$ instead,
 which is the unique $C_4$-Galois subextension of $\QQ(\zeta_{16})$ over $\QQ$.
 If we denote for $p=2$ or $p\equiv 1 \mod 4$ the character of $L/\QQ$ by $\chi$,
 these define $C_4$-Galois covers of $\Mscr_S$ by pullback and we abuse
 notation and write $\chi$ also for these covers. The cup product
 $[(\chi,\Delta)_4]$ in $\Br'(\Mscr_{\ZZ_P})$ is a class such that
 $2[(\chi, \Delta)_4] = [(p, \Delta)_2]$ and thus has exact order $4$.
 It follows that the classes
 $[(\chi,\Delta)_4]$ give a basis of the $4$-torsion of $\overline{\Br}(\Mscr_{\ZZ_P})$.
\end{remark}

\section{The Brauer group of $\Mscr$}

%
%
%

In this section, we will complete the computation of $\Br(\Mscr)$. In the last
section, we saw that 
$$\Br(\Mscr[\tfrac12]) \cong \ZZ/2 \oplus \ZZ/2 \oplus \ZZ/4$$
with generators $\alpha = [(-1,-1)_2]$, $\beta = [(-1, \Delta)_2]$ and $\frac12 \gamma$ for $\gamma = [(2, \Delta)_2]$. 
We will study the ramification of these classes for certain elliptic curves and the reader can find more information about
these curves in \emph{The $L$-functions and Modular Forms
Database}~\cite{lmfdb}.

As above, we will write $\binom{a,b}{2}\in\mu_2(\QQ_2)$ for the Hilbert symbol in $\QQ_2$ at
the prime $2$. Hence, $\binom{a,b}{2}=\pm 1$. Recall from
Proposition~\ref{prop:Hilbert} that if
$\chi\in\H^1(\QQ_2,C_2)\iso\H^1(\QQ_2,\mu_2)$ corresponds to a unit
$v\in\Gm(\QQ_2)/2$, then the degree $2$ cyclic algebra
$(\chi,u)$ has class $\binom{u,v}{2}$ in
$\Br(\QQ_2)[2]\iso\ZZ/2\iso\mu_2(\QQ_2)$. Since $\Br(\ZZ_2)=0$, the Hilbert
symbol measures the ramification along $(2)$ in $\Spec\ZZ_2$.

\begin{proposition}\label{prop:2ram}
    Every non-zero linear combination of $\alpha,\beta,\tfrac12\gamma$ is ramified 
    along $(2)$, so these linear combinations are not
    in the image of $\Br(\Mscr)\rightarrow\Br(\Mscr[\tfrac12])$.
\end{proposition}

\begin{proof}
	It suffices to prove this for all seven nonzero linear combinations of
	$\alpha,\beta,\gamma$. Indeed, if all these linear combinations are ramified, then
	any linear combination $r\alpha+s\beta+\tfrac12\gamma$ is ramified as well.
	As explained in the introduction around the diagram \eqref{eq:BrauerExtending}, it suffices to construct for each
	non-zero linear combination $\rho$ an elliptic curve
	$\Spec\ZZ_2\rightarrow\Mscr$ such that the pullback of $\rho$
	to $\Spec\QQ_2$ is non-zero in $\Br(\QQ_2)$. Indeed, $\Br(\ZZ_2)=0$.
	
	Let $$E_1:y^2+y=x^3-x^2,$$ the elliptic curve with Cremona label
	\texttt{11a3}. This curve has discriminant $-11$, which is a unit, so we
	get an elliptic curve over $\ZZ_2$, and two associated
	Brauer classes, $(2, -11)$ and $(-1, -11)$ over $\QQ_2$. We can ask what the
	ramification is. 
	The Hilbert symbol in this case is computed as
	follows~\cite{serre-course}*{Chapter~III}. Given $a=2^\alpha u$ and
	$b=2^\beta v\in\Gm(\QQ_2)$, where $u,v\in\Gm(\ZZ_2)$, we have
	$$\binom{a,b}{2}=(-1)^{\epsilon(u)\epsilon(v)+\alpha\omega(v)+\beta\omega(u)},$$
	where $\epsilon(u)\equiv\tfrac{u-1}{2}$ and
	$\omega(u)\equiv\tfrac{u^2-1}{8}$. 
	
	Hence,
	$$\binom{2, -11}{2}=(-1)^{\omega(-11)}=(-1)^{15}=-1.$$
	Hence, $(\Delta,2)$ is ramified at $2$. Similarly, 
	$$\binom{-1,-11}{2}=(-1)^{\epsilon(-1)\epsilon(-11)}=(-1)^{6}=1.$$
	
	The curve $$E_2:y^2+xy=x^3-2x^2+x$$ has Cremona label \texttt{15a8} and
	discriminant $-15$. This time, the Hilbert symbols are $$\binom{2,-15}{2}=1$$ and
	$$\binom{-1, -15}{2}=1.$$
	
	The curve $$E_3:y^2+xy+y=x^3-x^2$$ has Cremona label \texttt{53a1} and
	discriminant $-53$. In this case, the Hilbert symbols are
	$$\binom{2,-53}{2}=-1$$ and $$\binom{-1, -53}{2}=-1.$$
	
	Now, let $\rho=\alpha+k\beta+m\gamma$ with $k$ and $m$ integers. The
	elliptic curve $E_2$ gives a point $x:\Spec\QQ_2\rightarrow\Mscr$ where
	$v_2(x^*\rho)=-1$. It follows that all classes of the form $\rho$ are
	ramified along $(2)$. Now, $E_3$ proves that $\beta$ and $\gamma$ are
	ramified along $(2)$. Finally, $E_1$ proves that $\beta+\gamma$ is ramified
	along $(2)$.
\end{proof}

\begin{theorem}
    $\Br(\Mscr) = 0$.
\end{theorem}

\begin{proof}
    This follows from Corollary~\ref{cor:nop}, Theorem \ref{thm:Br3},
    and Propositions \ref{prop:brQ} and \ref{prop:2ram}.
\end{proof}

\section{The Brauer group of $\Mscr$ over $\FF_q$ with $q$ odd}

As another application of our methods, we compute the Brauer group
of $\Mscr_{\FF_q}$ when $q=p^n$ and $p$ is an odd prime. There is remarkable
regularity in this case, which is possibly surprising based on what happens
for number fields.

\begin{theorem}
    Let $q=p^n$ where $p$ is an odd prime. Then,
    $\Br(\Mscr_{\FF_q})\iso\ZZ/12$.
\end{theorem}

\begin{proof}
    Recall that $\Br(\FF_q)=0$. Thus, by Theorem~\ref{thm:brm}, there is an
    extension
    $0\rightarrow\FF_q^\times/2\rightarrow{_2\Br(\Mscr_{\FF_q})}\rightarrow\FF_q^\times/2\rightarrow
    0$. Since $q$ is odd, $\FF_q^\times/2\iso\ZZ/2$. The remainder of Section~\ref{sec:2torsion}, especially
    Remark~\ref{rem:4cyclic}, implies that the extension is non-split so that in fact $_2\Br(\Mscr_{\FF_q})\iso\ZZ/4$.

    Now, let $\ell\geq 3$ be a prime, which we \emph{do not} assume is
    different from $p$. The possible terms contributing to
    $_\ell\Br(\Mscr_{\FF_q})$ in the Leray--Serre spectral
    sequence for $\Mscr(2)_{\FF_q}\rightarrow\Mscr_{\FF_q}$ in $\Gm$
    cohomology, besides $_\ell\Br(X_{\FF_q})^{S_3}$, are
    $\H^1(S_3,_\ell\Pic(\Mscr(2)_{\FF_q})$ and $\H^2(S_3,_\ell\Gm(\Mscr(2)_{\FF_q}))\iso\H^2(S_3,_\ell\Gm(X_{\FF_q}))$.
    The first of these is zero since
    $_\ell\Pic(\Mscr(2)_{\FF_q})\iso{_\ell\Pic(X_{\FF_q})}=0$ for $\ell$ odd.
    The second has no odd primary torsion. This follows from the exact sequence
    $0\rightarrow\Gm(\FF_q)\rightarrow\Gm(X_{\FF_q})\rightarrow\tilde{\rho}\otimes\ZZ\rightarrow
    0$ together with Lemmas~\ref{lem:htriv} and~\ref{lem:hrho}.

    Thus, we see that
    $_\ell(\Br(\Mscr_{\FF_q}))\iso{_\ell\Br(X_{\FF_q})}^{S_3}$ for $\ell$ odd.
    So, it suffices to compute the Brauer group of $X_{\FF_q}$ as an $S_3$-module. In general, our
    argument in the rest of the paper relies fundamentally on
    Lemma~\ref{lem:supportsrho},
    which requires $\ell$ to be invertible to analyze the ramification map as
    in Proposition~\ref{prop:puritycalc}. However, in this case, $X_{\FF_q}$
    is a curve over a finite field, so the ramification theory simplifies
    drastically. Consider the commutative diagram of exact ramification sequences due
    to~\cite{grothendieck-brauer-3}*{Proposition 2.1}
    \begin{equation*}
        \xymatrix{
            0\ar[r] &   \Br(\PP^1_{\FF_q})\ar[r]\ar[d]    &
        \Br(\eta)\ar[r]\ar[d]^=  &
        \bigoplus_{x\in(\PP^1_{\FF_q})^{(1)}}\QQ/\ZZ\ar[r]\ar[d]&\QQ/\ZZ\ar[r]&  0\\
            0\ar[r] &   \Br(X_{\FF_q})\ar[r]                & \Br(\eta)\ar[r] &
        \bigoplus_{x\in (X_{\FF_q})^{(1)}}\QQ/\ZZ
        }
    \end{equation*}
    where $\eta$ is the generic point of $X_{\FF_q}$. Note that the exactness
    at the right in the top sequence is due to the fact
    (see~\cite{gille-szamuely}*{Corollary~6.5.4}) that
    $\bigoplus_{x\in(\PP^1_{\FF_q})^{(1)}}\QQ/\ZZ\rightarrow\QQ/\ZZ$ is given
    by summation in $\QQ/\ZZ$.
    Since $\Br(\PP^1_{\FF_q})=0$
    by~\cite{grothendieck-brauer-3}*{Remarques~2.5.b}, we see that
    $\Br(X_{\FF_q})\subseteq\Br(\eta)$ is the subgroup consisting of classes
    ramified only at $0,1,\infty$. Using the top row of the diagram, it follows
    that $\Br(X_{\FF_q})$ fits into an exact sequence
    $$0\rightarrow\Br(X_{\FF_q})\rightarrow\bigoplus_{0,1,\infty}\QQ/\ZZ\rightarrow\QQ/\ZZ\rightarrow 0,$$
    from which it follows that
    $$\Br(X_{\FF_q})\iso\tilde{\rho}\otimes\QQ/\ZZ.$$
    By Lemma~\ref{lem:hrho}, we find that ${_\ell\Br(X_{\FF_q})}^{S_3}=0$ if
    $\ell\geq 5$ and ${_3\Br(X_{\FF_q})}^{S_3}\iso\ZZ/3$. This proves the
    theorem.
\end{proof}

\begin{remark}
	We can again be more specific about the Azumaya algebras representing the
    classes in $\Br(\Mscr_{\mathds{F}_q})$ with $q$ odd. Let $T$ be a
    $\FF_q$-scheme (or stack) with an elliptic curve $E$ of discriminant
    $\Delta$. Let $\chi_m$ be the pullback of a character of the Galois
    extension $\FF_{q^m}$ of $\FF_q$ to $T$. We claim that there is a generator
    $a\in \Br(\Mscr_{\mathds{F}_q})$ such that the pullback of $a$ to $\Br(T)$ agrees
    with $[(\chi_{12},\Delta)]$. Informally, $a = [(\chi_{12},\Delta)]$ in the
    universal case $T = \Mscr_{\FF_q}$, where more precisely we should replace
    here $\Delta$ by the $\mu_{12}$-torsor corresponding to $\Delta \in \Gamma(\lambda^{\tensor
    12})$ via Construction \ref{const:root}.
	
	First we consider the $3$-torsion. The proof of Lemma \ref{lem:3section}
    applies here to show that $[(\chi_3, \Delta)]$ is indeed the pullback of a
    generator of $\Br(\Mscr_{\mathds{F}_q})[3]$.
	Moreover, Proposition \ref{prop:2cyclic} implies that the unique
    $2$-torsion element $6a \in \Br(\Mscr_{\mathds{F}_q})$ pulls back to
    $[(\chi_2, \Delta)]$ as $\FF_{q^2}$ agrees with $\F_q[\sqrt{x}]$ for an
    arbitrary non-square $x$ in $\FF_q$.
	As in Remark \ref{rem:4cyclic} we see that $6[(\chi_{12}, \Delta)] =
    [(\chi_2,\Delta)] \neq 0$ and $4[(\chi_{12}, \Delta)] = [(\chi_3,\Delta)]
    \neq 0$. Thus, $[(\chi_{12}, \Delta)]$ is indeed a generator of $
    \Br(\Mscr_{\mathds{F}_q}) \cong \ZZ/12$.
\end{remark}

Finally, we also treat the easier case of an algebraically closed base. 

\begin{proposition}
    Let $k$ be an algebraically closed field of characteristic not $2$. Then $\Br(\Mscr_k) = 0$. 
\end{proposition}

\begin{proof}
    By Theorem~\ref{thm:brm}, ${_2}\Br(\Mscr_k) = 0$. As in the last proof we
    see that $_\ell(\Br(\Mscr_{k}))\iso{_\ell\Br(X_k)}^{S_3}$ for $\ell$ an
    odd prime. By Tsen's theorem, $\Br(\eta)$ vanishes for $\eta$ the generic
    point of $X_k$. By \cite{grothendieck-brauer-3}*{Proposition 2.1} we obtain
    that $\Br(X_k) = 0$.
\end{proof}

\begin{remark}
	In an earlier version of this paper we suggested using the $\GL_2(\ZZ/3)$-cover $\Mscr(3)\rightarrow\Mscr_{\ZZf{3}}$ to determine $\Br(\Mscr_k)$ also for algebraically closed fields $k$ of characteristic $2$. Combining this approach with a new idea,
    Minseon Shin has proved in the meanwhile in~\cite{shin} that $\Br(\Mscr_k)\iso\ZZ/2$ for such fields $k$. 
\end{remark}


\end{document}